\documentclass{amsart}

\usepackage{amssymb}
\usepackage{amsmath}
\usepackage{amsthm}
\usepackage{amsbsy}
\usepackage{amscd}
\usepackage{bm}
\usepackage{graphicx}
\usepackage{tikz-cd}
\usepackage{enumerate}
\usepackage{enumitem}
\usepackage{float}

\theoremstyle{plain}
\newtheorem{theorem}{Theorem}[section]
\newtheorem{lemma}[theorem]{Lemma}

\newtheorem{proposition}[theorem]{Proposition}

\theoremstyle{definition}
\newtheorem{definition}[theorem]{Definition}

\theoremstyle{remark}
\newtheorem{remark}[theorem]{Remark}

\newcommand{\C}{\mathbb{C}} 
\newcommand{\N}{\mathbb{N}}

\begin{document} 

\title{Embeddings of   $4$--manifolds in $\mathbb{C}P^3$}

\author{Abhijeet Ghanwat}

\address{ Chennai Mathematical Institute\\
          H1, SIPCOT IT Park,\\
		Kelambakkam (Chennai), India.}

\email{abhijeet@cmi.ac.in}

\author{Dishant M. Pancholi}
\address{ The Institute of Mathematical Sciences\\
          No. 8, CIT Campus, Taramani, \\
	   Chennai 600113, India.}
\email{dishant@imsc.res.in}

\date{\today}

\subjclass{Primary ; Secondary }

\begin{abstract}
In this article we show that every closed orientable smooth $4$--manifold admits a smooth embedding in the complex projective $3$--space. We also provide a new proof of embeddings of $4$--manifolds in 
$\mathbb{R}^7.$
\end{abstract}
\maketitle

\section{Introduction}
A basic question in the field of geometric topology  which 
concerns embeddings of manifolds,  can be  stated  as follows: Given a pair of manifolds $M$ and $N,$ how many smooth embeddings of $M$ exist in  $N?$

Detailed investigations in this regard have  led to the discovery of  interesting invariants of manifolds. One of the earliest seminal results in this context 
is due to  H. Whitney who showed that every closed manifold of dimension $n$ admits an embedding in $\mathbb{R}^{2n}.$ Subsequently, this result has been extensively generalized. Most notably, M. Hirsch  showed \cite{Hi} that every closed orientable 
odd--dimensional manifold $M^{2n-1}$ admits a smooth embedding in 
$\mathbb{R}^{4n-3}.$  This result, together with those by C.T.C Wall and V. Rokhlin implies that every closed $3$--manifold admits an embedding in $\mathbb{R}^5.$

For closed even dimensional  manifolds, combining 
results of A. Haefliger~\cite{Ha}, A. Haefliger and M. Hirsch~\cite{HH}, and W. Massey and F. Peterson~\cite{MP}, one knows that every orientable $n$--manifold embeds in $\mathbb{R}^{2n-1}$ when  $n>4,$ and  if $n$ is not a power of two, then  every n-manifold embeds in  
$\mathbb{R}^{2n-1}.$
For   $4$--manifolds it was shown by M. Hirsch 
 \cite{Hi1} and C. T. C. Wall\footnote{ M. Hirsch has mentioned in \cite{Hi1} that
 C. T. C. Wall had independently proved this result.}  that every orientable PL $4$--manifold admits a PL embedding in $\mathbb{R}^7.$

It is usually possible to construct  an invariant of a manifold $M$ using its embeddings in a manifold $N,$  provided that $(1)$ the topology of $N$ is  relatively simple and  $(2)$ the co-dimension of the embedding of  $M$ in $N$ is small. 
The importance of these two conditions is evident even from the examples 
of  embeddings of surfaces. We recall that there exists an embedding of a  closed smooth surface $\Sigma$ in $\mathbb{R}^3$ if and only if  
$\Sigma$ is orientable. This clearly shows that the orientability of a smooth closed surface  can be captured by its embeddability in  Euclidean $3$--space. 
Further, the  embeddability  of every closed surface  in 
$\mathbb{R}^4$ demonstrates  the importance of lower co-dimension of  embeddings, 
while the fact that  $\mathbb{R}P^3 \# \mathbb{R}P^3$ admits an embedding of every closed surface shows the need for   a relatively  simple topology for the target space.

It was shown by S. Cappell and J. Shaneson \cite{CS} that a smooth $4$--manifold admits a smooth embedding in $\mathbb{R}^6$ if and only if it admits a spin structure.  We know that a closed orientable $4$--manifold is spin if and only if the second Stiefel-Whitney class $w_2(M)$ is zero.  In particular, this implies that 
$\mathbb{C}P^2$ does not smoothly embed in $\mathbb{R}^6.$
In this article, we  investigate whether there exist  topologically simple closed $6$--dimensional manifolds which admit embeddings of all smooth $4$--manifolds. 

Two important classes of closed orientable smooth $4$--manifolds are symplectic $4$--manifolds and smooth algebraic surfaces.
Their embeddings in various complex projective spaces have been extensively examined
(see, for instance \cite{BHPV},\cite{Do1}, and \cite{Do}), and the question of their embeddability in $\mathbb{C}P^3$ is very important. Furthermore, the topology of $\mathbb{C}P^3$ is very simple and $\mathbb{C}P^2$ naturally embeds in $\mathbb{C}P^3.$ 
We therefore investigate embeddings of  $4$--manifolds in $\C P^3$ and establish the following:

 \begin{theorem}\label{thm:embedding_4-manifolds}
Every closed orientable smooth $4$--manifold admits a smooth embedding in 
$\mathbb{C}P^3.$ 
\end{theorem}

To the best of our knowledge, Theorem~\ref{thm:embedding_4-manifolds} and 
Theorem~\ref{thm:embedding_in_CP^2_times_CP^1}, which establishes embedding of $4$--manifolds in certain $6$--manifolds of the type $N \times \C P^1$,  are   the only  results 
establishing the existence of closed $6$--manifolds in which all orientable smooth 
$4$--manifolds embed.  

The central idea for the proof of Theorem~\ref{thm:embedding_4-manifolds} is drawn
from a well--known fact that given a projective embedding of  a smooth algebraic surface, the standard \emph{Lefschetz pencil} of the complex projective space generically induces a Lefschetz pencil structure on the surface. It was established by I. Baykur and  O. Saeki  \cite{BO} that every smooth $4$--manifold admits a  \emph{simplified broken Lefschetz fibration} (SBLF), 
which can be regarded as a natural  generalization of the Lefschetz pencil for an arbitrary smooth $4$--manifold.  This decomposition allows us to express any smooth $4$--manifold as a singular fiber bundle over $\C P^1$ with a finite number of
 \emph{Lefschetz singularities} and a unique  \emph{fold singularity}. 
The advantage of this decomposition is that we can  associate 
with any smooth $4$--manifold certain  data which comprise two constituents. 
These are: (1) an element of the \emph{mapping class group} of a closed  orientable surface of genus $g$  expressed as a product of \emph{Dehn twists}, corresponding to 
Lefschetz singularities, and (2) a  round handle attachment corresponding to  the fold singularity.

Let us now briefly outline the argument establishing Theorem~\ref{thm:embedding_4-manifolds}.  We 
need Theorem~\ref{thm:embedding_in_CP^2_times_CP^1}  to establish Theorem~\ref{thm:embedding_4-manifolds}. Hence, we begin by outlining a proof of 
Theorem~\ref{thm:embedding_in_CP^2_times_CP^1}.

Consider  any closed orientable $4$--manifold $N$ which admits an embedding of a Hopf link
which is \emph{separable} in the sense of Definition~\ref{def:separable_Hopf_link}, by which we mean 
that $N$ admits a handle decomposition that satisfies the following property:  the 
boundary of a $0$--handle has a Hopf link, which is a slice in the complement of the $0$--handle. 
In the following discussion we fix one such $4$--manifold $N.$

Given a closed orientable smooth $4$--manifold $M$,  consider 
the manifold $M$ together with any given 
SBLF. The first step is to produce an embedding $f$ of $M$
in $N \times \C P^1$ such that the trivial  fibration 
$\pi_2: N \times \C P^1 \rightarrow \C P^1$ of $N \times \C P^1$ induces the given  SBLF.

The three important steps for constructing the embedding $f$ are the following: 
In the first step, using an  appropriate generalization of techniques 
from \cite{PPS}, and  a specific local embedding model
for a given  Lefschetz singularity, we provide an embedding of genus $g+1$ 
\emph{Lefschetz sub--fibration} over a disk $\mathbb{D}^2$ in 
$N \times \mathbb{D}^2,$
which is associated with the given SBLF. 
This embedding is   such that the trivial product fibration 
$\pi_2: N \times \mathbb{D}^2 \rightarrow \mathbb{D}^2$ induces  the given \emph{Lefschetz fibration}. This is the most important  step in the proof, and is
detailed  in Section~\ref{sec:lef_fib_embedding}.  In fact, in Section~\ref{sec:lef_fib_embedding}  we show how to embed any Lefschetz fibration over a disk
or $\C P^1$ in a trivial fibration over  $\C P^1$ with fiber $N.$

Next, we use a local  embedding model for fold singularities to produce an
embedding of a sub-manifold $(\widetilde{M},\partial \widetilde{M}) \subset M$ (having two disjoint boundary components) 
in $N \times I \times \mathbb{S}^1.$ This embedding is constructed such that it
agrees with the embedding in the first step near one of the boundary
components of $\widetilde{M},$ and is  a trivial fibration $\Sigma_g \times S^1$ 
near the other boundary component of $\widetilde{M}.$ Here, $\Sigma_g$ denotes
a surface of genus $g.$  This provides us with a fiber preserving  embedding of 
$M  \setminus \Sigma_g \times \mathbb{D}^2$ in 
$N \times \mathbb{D}^2.$  Finally, we extend the embedding 
of $M \setminus \Sigma_g \times \mathbb{D}^2$ 
in $N \times \mathbb{D}^2$  using an  embedding of $\Sigma_g \times \mathbb{D}^2$ in $N \times \mathbb{D}^2$ to obtain  the  embedding  $f:M \hookrightarrow N \times \C P^1.$  These two steps are discussed in Section~\ref{sec:embed_in_product}. Embeddings of $M$ in $N \times \C P^1$ is
the content of Theorem~\ref{thm:embedding_in_CP^2_times_CP^1}. 
 Theorem~\ref{thm:embedding_in_CP^2_times_CP^1} immediately implies
 Theorem~\ref{thm:embedding_in_R^7} which establishes  embeddings of  smooth closed orientable $4$--manifolds in $\mathbb{R}^7.$

Having outlined a proof of Theorem~\ref{thm:embedding_in_CP^2_times_CP^1}, let us
now discuss how to establish embeddings of $4$--manifolds in $\C P^3$ as claimed in
Theorem~\ref{thm:embedding_4-manifolds}. Given a smooth, orientable, closed $4$--manifold,   we first  consider the manifold $M \# \C P^2 \# \overline{\C P^2}$ together with a specific SBLF.  
Next, we notice that the  the \emph{blow-up}  of $\C P^3$ along $\C P^1$ is 
a fiber bundle over $\C P^1$ with fiber $\C P^2$ with   the property that
the fiber bundle is trivial in the complement of the \emph{exceptional divisor}.

We  embed $M \# \C P^2 \# \overline{\C P^2} $ the blow-up of $\C P^3$  using this specific SBLF  by observing that $\C P^2$ admits a separable Hopf link, and hence a slight generalization of
the argument necessary to establish Theorem~\ref{thm:embedding_in_CP^2_times_CP^1} allows
us to embed $M\# \C P^2 \# \overline{\C P^2}$ Next, we note that the \emph{blow-up} of $\C P^3.$
Further, by ensure certain intersection property of the fiber of the specific SLBF  , we ensure that the embedding of $M \# \C P^2  \#\overline{\C P^2}$  constructed   is  such that
when we \emph{blow-down} the blow-up of $\C P^3,$ we produce a $\C P^3$ that 
has $M$ as its  embedded sub-manifold. The construction of the
specific SBLF, blow-up and blow-down procedures, and the proof of 
Theorem~\ref{thm:embedding_4-manifolds} are discussed  in the final section.

The mathematical preliminaries to carry out these steps are given in   Sections \ref{sec:review_blf} and \ref{sec:review_mcg}. In particular, we discuss relevant aspects of  \emph{broken Lefschetz fibrations} in Section~\ref{sec:review_blf}, and of
mapping class groups in Section~\ref{sec:review_mcg}.  

 Finally, a few remarks on conventions used in this article.  By  a  manifold we mean 
a compact orientable  manifold  with or without boundary. We denote manifolds by capital letters  $M,$ $ N,$ etc. When we need to emphasis that we are working with a manifold with boundary, we use the notation $(M, \partial M)$ consisting of  the pair  $M$ and the boundary $\partial M$ of $M.$ 
As usual, the notation  $\Sigma$ or $ \Sigma_g$ is 
used for  denoting a closed orientable surface, with $g$ indicating  the genus.

\subsection{Acknowledgment} \mbox{} 
 Dishant M. Pancholi is thankful to the Simon's Foundation and
ICTP, Trieste, Italy, for  award  of the Simons  Associateship, which allowed
him to travel to ICTP, Trieste, Italy, where a part of the work related to this
article was carried out. We are thankful to Prof. Yakov Eliashberg for contact encouragement 
and support. We are also thankful to Prof. S. Lakshmibala for suggestions regarding 
the presentation.

\section{Review of Broken Lefschetz fibrations} \label{sec:review_blf}

Broken Lefschetz fibrations  (BLF) were introduced by D. Auroux, S. K.  Donaldson, and L. Katzarkov in \cite{ADK}. These are generalized Lefschetz fibrations. 
I. Baykur~\cite{Ba} established  that every smooth orientable $4$--manifold admits a broken Lefschetz fibration. The purpose of this section is to review few definitions
and result related  to BLF. We refer to  \cite{Ba} and \cite{BO} for a detailed
discussion on BLF.  Let us begin by recalling  the definition of Lefschetz singularity.

\begin{definition}[Lefschetz singularity] Let $M$ be an oriented $4$--manifold
and $\Sigma$  an oriented surface.
Let $f: M \rightarrow \Sigma$ be a smooth map. A point $x \in M$ is said to 
have a Lefschetz singularity at $x$ for the map $f,$ provided that there is an orientation preserving
parameterization  $\phi: U \subset M \rightarrow \mathbb{C}^2,$ and an orientation preserving 
parameterization $\psi: V \subset \Sigma \rightarrow \mathbb{C}$ such that
the following properties are satisfied:

\begin{enumerate}
\item $x \in U,$ and $\phi(x) = (0, 0) \in \mathbb{C}^2.$
\item $f(x) \in V,$ and $\psi(f(x)) = 0 \in \mathbb{C}.$
\item  For the map $g:\mathbb{C}^2 \rightarrow \mathbb{C}$ given by 
$g(z_1, z_2) = z_1.z_2,$ the following diagram commutes:

\begin{center}

\begin{tikzcd}
 U \arrow[r, "\phi"] \arrow[d, "f"] &  \mathbb{C}^2  \arrow[d, "g"]  \\
V \arrow[r,  "\psi" ] & \mathbb{C}. 
\end{tikzcd}
\end{center}
\end{enumerate}\label{def:lef_sing}
\end{definition}

\begin{remark}\mbox{}
\begin{enumerate}[label=(\alph*)]

\item Observe that both $M$ as well as $\Sigma$ can have non-empty boundary, however,
it follows from Definition~\ref{def:lef_sing} that the critical point  $c$ 
belongs to the  interior  $\mathring{M}$ of $M,$ and  $f(c) \in \mathring{\Sigma}.$

\item
In case we do not put any condition regarding  preservation of  orientations  by the parameterization around $x$ and $f(x)$ in  Definition~\ref{def:lef_sing} above, the singularity is termed as \emph{achiral Lefschetz singularity}.

\item Let $f:M \rightarrow S$ be a map with an isolated Lefschetz singularity
at $c \in M.$ It is well known that the fiber over $f(c)$ is obtained by pinching
a simple closed curve $\gamma$  on  nearby smooth fiber $\Sigma_g$ to a point.
The curve $\gamma$ is known as  a \emph{vanishing cycle}.

\item If we take a small closed disk $\mathbb{D}$ around $f(c)$ not containing any other critical value, then the $f^{-1}(\partial \mathbb{D})$ is a mapping torus over
the smooth fiber $\Sigma_g$ with monodromy a  positive Dehn twist
along the vanishing cycle $\gamma.$ In case of an achiral Lefschetz singularity, 
the monodromy could be a positive or a negative Dehn twist along $\gamma.$

\end{enumerate}
\end{remark}

Next, we  recall the definition of \emph{$1$--fold singularity}.

\begin{definition}[$1$--fold Singularity]
 Let $M$ be an oriented $4$--manifold, and let $\Sigma$ be an oriented surface.
Let $f: M \rightarrow \Sigma$ be a smooth surjective map. A point $x \in M$ is said to have a 
$1$--fold singularity at $x$ provided there is an orientation preserving
parameterization  $\phi: U \subset M \rightarrow \mathbb{R}^4,$ and an orientation preserving 
parameterization $\psi: V \subset \Sigma \rightarrow \mathbb{R}^2$ such that
the following properties are satisfied:

\begin{enumerate}
\item $x \in U,$ and $\phi(x) = (0, 0,0,0) \in \mathbb{R}^4.$
\item $f(x) \in V,$ and $\psi(f(x)) = (0,0) \in \mathbb{R}^2.$
\item  For the map $h:\mathbb{R}^4 \rightarrow \mathbb{R}^2$ given by 
$h(t, x_1, x_2, x_3) =  (t, -x_1^2 + x_2^2 +x_3^2),$ the following diagram commutes:

\begin{center}
\begin{tikzcd}
 U \arrow[r, "\phi"] \arrow[d, "f"] &  \mathbb{R}^4  \arrow[d, "h"]  \\
V \arrow[r,  "\psi" ] & \mathbb{R}^2. 
\end{tikzcd}
\end{center}
\end{enumerate}
\end{definition}\label{def:1-fold_sing}

\begin{remark}\mbox{}
\begin{enumerate}[label=(\alph*)]
\item If a map $f:M \rightarrow \Sigma$
has a  $1$--fold singularity at $x$, then $x \in \mathring{M},$ and $f(x) \in \mathring{\Sigma}.$

\item
When the map $h$ in the definition of $1$-fold singularity  is allowed to
have the local model: 
$$(t,x_1, x_2, x_3) \rightarrow (t, \pm x_1^2 \pm x_2^2 \pm x_3^2),$$
 the singularity is termed as a \emph{fold singularity}. In this article,  
we will only need the local model around  $1$--fold singularity. 

\item A local singularity model for a smooth function of the  form:
$$(t, x_1, x_2,x_3) \rightarrow (t, x_1^3 + t x_1 \pm x_2^2  \pm x_3^2)$$ 
is known as a \emph{cusp singularity}.

\end{enumerate}
\end{remark}

We are now in a position to recall the notion of a broken Lefschetz fibration (BLF).

\begin{definition}[Broken Lefschetz fibration]
Let $M$ a smooth oriented $4$--manifold. By a broken Lefschetz fibration of $M$
we mean a smooth map $f: M \rightarrow \C P^1$ such that $f$ has only
$1$-fold or Lefschetz singularity. 
\end{definition}

\begin{remark}\mbox{}
\begin{enumerate}[label=(\alph*)]
\item
 Given a BLF $f:M \rightarrow \C P^1,$ the inverse image $f^{-1}(y)$ for any
 regular value $y$ is called a fiber of BLF.

\item Generically,  the image set of a $1$--fold singularity on
$\Sigma$ is  an immersed circle in $\mathring{\Sigma}.$

\end{enumerate}
\end{remark}

A BLF without $1$--fold singularity is called  a Lefschetz fibration. These singular fibrations are extremely useful in algebraic geometry~\cite{GH} and symplectic geometry~\cite{Do}.
Let us now  formally define a  Lefschetz fibration.

\begin{definition}[Lefschetz fibration]
Let $M$ be a smooth oriented $4$--manifold. A smooth map $f:M \rightarrow \Sigma$,
where $\Sigma$ is an oriented surface, having its singular points modeled 
only on Lefschetz singularities is called a  Lefschetz fibration
of $M.$  
\end{definition}\label{def:lef_fibration}

\begin{remark}\mbox{}
\begin{enumerate}[label=(\alph*)]

\item Unlike a fiber bundle or Lefschetz fibration, the fibers of a BLF are typically not
diffeomorphic. In fact,  the $1$-fold singularity in the definition of BLF corresponds to a round $1$--handle attachment~\cite{GK, BO}. Hence, if BLF
has points having fold singularity, then the genus of fibers change as we cross
the image of an immersed circle coming from a $1$--fold singularity.

\item The fibers of BLF need not be connected. However, it can be shown
that every $4$--manifold admits a BLF with  connected fibers having genus at least 
$2.$ This follows from  ~\cite[Theorem 1.1]{Ba}.

\end{enumerate}
\end{remark}

Observe that a BLF provides us a decomposition of a smooth manifold into 
simple pieces. A more simplified form of this decomposition of smooth $4$--manifold
is what we will need for this article. This simplification was introduced by
I. Baykur and  O.  Saeki in \cite{BO}. This decomposition is known as a simplified
broken Lefschetz fibration. Let us recall the definition of this:

\begin{definition}[Simplified broken Lefschetz fibration  (SBLF)]\label{def:sblf}
 Let $f:M \rightarrow \C P^1$ be a BLF. We say that this BLF is a simplified 
 broken Lefschetz fibration (SBLF) provided the function $f$ satisfies the following
 additional properties:
 
\begin{enumerate}
\item The set $Z_f$ of all $x \in M$ admitting a $1$-fold singularity model is 
connected.

\item All fibers are connected. 

\item The map $f$ is injective when restricted to $Z_f$ as well as when restricted to
the set, $C_f,$ of Lefschetz singular points. 
\end{enumerate}
\end{definition}

The definition of SBLF motivates the definition of the following:

\begin{definition}[Simplified Lefschetz fibration (SLF)]\label{def:slf}
Let $M$ be a smooth oriented $4$--manifold. Let $\Sigma$ be $\mathbb{C}P^1$ or  a
$2$--disk $\mathbb{D}^2.$ A Lefschetz fibration $f:M \rightarrow \Sigma$ is said to be simplified 
Lefschetz fibration provided all the critical values of $f$ in $\Sigma$ are isolated, and for 
any regular value $y \in \Sigma,$ the fiber $f^{-1}(y)$ is connected.
\end{definition}

\begin{remark}\mbox{}

\begin{enumerate}[label=(\alph*)]
\item A SBLF having no fold singularity is a SLF.

\item Observe that the definition of  SBLF implies that there exists a disk $\mathcal{D}$ contained in $\C P^1$ such that  every $y \in \mathcal{D}$ is a regular value, and the genus of the fiber over $y$ is minimum among all fibers of SBLF.  We call this
fiber \emph{lower genus fiber}.

\item Topologically, the unique $1$--fold singularity of SBLF corresponds to adding 
$1$--handle to a circle worth of lower genus fibers over $\partial \mathcal{D}.$ 
This corresponds to an attachment  of a round $1$--handle to $f^{-1}(\mathcal{D})$ 
such that a generic fiber of SBLF over $\C P^1 \setminus \overline{\mathcal{D}}$ 
has genus one more than the fibers over $\mathcal{D}.$
\end{enumerate}

\end{remark}

%Observe that SBLF implies that there exists a  disk $\mathbb{D}$ contained in 
%$\C P^1$ such that  every $y \in \mathbb{D}$ is a regular value and 

In \cite{BO}, it was shown that every orientable smooth $4$--manifold admits a
SBLF. 
%

 %\begin{definition}[Indefinite fibration]
%Let $M$ be a compact manifold. A smooth map $f$ from $M$ to a surface $\Sigma$
%is said to be an \emph{indefinite} fibration provided it has  only fold or 
%cusp  singularities.
%\end{definition}

\begin{theorem}[I. Baykur, O. Saeki: Theorem-1~\cite{BO}]\label{thm:existence_SBLF}
Given any generic map from a closed, connected, oriented, smooth 4-manifold
$X$ to $\mathbb{C}P^1,$  there are explicit algorithms to modify it to a SLBF. In particular, every closed orientable smooth $4$--manifold admits a SBLF.  Furthermore,  we can always construct
a SLBF on $M$ such that the genus of lower genus figure is bigger than $1.$
\end{theorem}

We would like to point out that Theorem~\ref{thm:existence_SBLF} is not stated as above in
~\cite{BO}. The statement regarding the lower bound on the genus of  a lower genus fiber is not explicitly mentioned in \cite[Theorem-1]{BO}. However, it follows from the application of 
\cite[Theorem-1]{BO} followed by \cite[Theorem-2]{BO}. For the sake of completeness, we 
discuss the proof of Theorem~\ref{thm:existence_SBLF}.

\begin{proof}
To begin with, recall that by a 
\emph{trisection} of a smooth
orientable closed $4$--manifold $M,$  one means a decomposition of $M$ into three $4$--dimensional
handle-bodies (thickening of a wedge of circles), meeting pairwise in $3$--dimensional handle-bodies, and all three $4$--dimensional handle-bodies intersect in a surface.  Trisections correspond to 
a Morse $2$--function on $M.$ If $k'$ is the number of indefinite folds for the Morse $2$--function
associated to a given trisection, and
$g'$ is the genus of the surface corresponding to  the common  intersections of three $4$--dimensional handle-bodies,
then one says that the $4$--manifold has a  $(g',k')$--trisection.

In order to produce a SBLF as stated in Theorem~\ref{thm:existence_SBLF}, we observe that given 
$M$, according to \cite[Theorem-1]{BO}, there exists a SBLF $f: M \rightarrow \mathbb{C}P^1.$
Let $g$ be the genus of lower genus fiber of the SLBF. If $g>1$, then we are through. In case, $g \leq 1$, we apply 
\cite[Theorem-2]{BO} to produce a $(g', k')$--trisection from the given SLBF $f: M \rightarrow \mathbb{C}P^1.$  According to \cite[Theorem-2]{BO}, we get a $(g',k')$--trisection 
with  $g' \geq  1.$

Next, we again apply the second part of \cite[Theorem-2]{BO} to produce from this trisection 
a new SBLF. Observe that according to \cite[Theorem-2]{BO}, the new SBLF has lower genus
fiber having its genus $g' +2.$ Since $g \geq 0,$ the Theorem follows.

\end{proof}

\section{Mapping class groups of surfaces}\label{sec:review_mcg}

In this section we review some results  related to 
 mapping class groups of  closed orientable surfaces.  Good references
 for the results discussed here  are \cite{BM} and \cite{Li}.
 Let us begin  by recalling the definition of the mapping class group.

\begin{figure}[h]
\centering
\includegraphics[scale=.8]{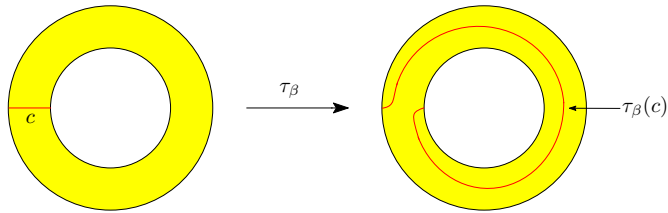}
\caption{The figure is a pictorial description of the Dehn twist 
$\tau_{\beta}$ restricted to the neighborhood $\mathcal{A}(\beta) = \mathbb{S}^1 \times [0,2 \pi].$ 
$\tau_{\beta}$  is given by $\tau_{\beta}(\theta, t) =(\theta + t, t)$ when restricted to
$\mathcal{A}(\beta).$ It sends the arc $c$ -- depicted as a red colored arc in the picture on the left of the figure -- to the arc $\tau_{\beta}(c)$ depicted in the picture on the right of the figure.} 

\label{fig:dehn_twist}
\end{figure}

\begin{definition}[Mapping class group]
Let $\Sigma$ be a closed orientable surface. By the mapping class group of $\Sigma ,$ we mean
the group of orientation preserving self diffeomorphisms of $\Sigma$ up to isotopy. 
\end{definition}

We denote the mapping class group of a surface $\Sigma$ by $\mathcal{M}CG(\Sigma).$
 Next, let us discuss the notion of a \emph{Dehn twist} along a simple closed curve embedded in a surface $\Sigma$. We refer \cite{BM} for a more detailed discussion on Dehn twists.

\begin{definition}[Dehn twist]
Let $\Sigma$ be an orientable surface. Let $\beta$ be a simple closed curve 
embedded in the interior of $\Sigma$. By a Dehn twist along $\beta$, we mean
a diffeomorphism which is identity outside an annulus
neighborhood  $\mathcal{A} (\beta)$ of $\beta$ in $\Sigma ,$ and is  given
by $\tau_{\beta}$ on 
$\mathcal{A}(\beta)$ when restricted to $\mathcal{A}(\beta),$ where $\tau_{\beta}$ is the
diffeomorphism of $\mathcal{A}(\beta)$ described in Figure~\ref{fig:dehn_twist}.
\end{definition}

\begin{figure}
\centering
\includegraphics[scale=1.2]{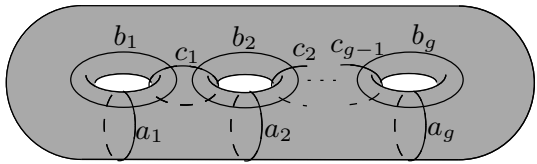}
\caption{Dehn twists along curves $a_i$'s ,$b_j$'s and $c_k$'s generate  the mapping class group of an orientable genus $g$ surface. }\label{fig:lick_gen}
\end{figure}

M. Dehn~\cite{De} and W. Lickorish~\cite{Li} independently  established that the mapping class group of an orientable genus $g$ surface $\Sigma_g$ is generated by Dehn twists along simple closed curves embedded in $\Sigma_g.$  W. Lickorish further  strengthened this result  in \cite{Li1}, to show that  the mapping class group of a closed orientable surface $\Sigma_g$ is generated by Dehn twists along the curves $a_i$'s , $b_j$'s 
and $c_k$'s as depicted in Figure~\ref{fig:lick_gen}. Following \cite{PPS}, we will 
call these curves as \emph{Lickorish generators}.

We end this section with a proposition which is a consequence of Lemma-3 established
in \cite{Li}. In order to state this proposition we need a few terminologies from
\cite{Li}. 

Let us regard an orientable  surface $\Sigma_g$ of genus $g$ as the boundary of a standard
handle-body $H_g.$ Here, a standard handlebody $H_g$ consists of  $g$  $1$--handles attached 
the unit $3$--ball in  $\mathbb{R}^3$ as depicted in  Figure~\ref{fig:standard_handle_body}.

\begin{figure}[ht]
\centering
\includegraphics[height=4.5cm,width=5cm]{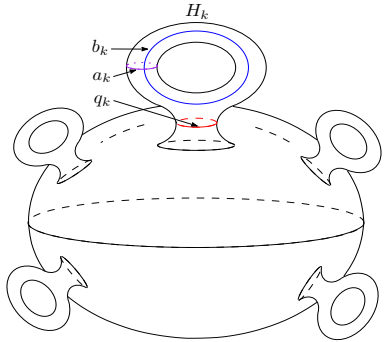}
\caption{The figure shows surface of genus $g$ embedded in $\mathbb{R}^3$ as a 
boundary of a genus $g$ handle-body considered as a unit ball with $g$ $1$--handles
attached to it.}
\label{fig:standard_handle_body}
\end{figure}

Consider a typical handle $H_k,$ as shown in Figure~\ref{fig:standard_handle_body}. 
Following \cite{Li}, we say that a simple closed curve $p$ 
\emph{does not meet} the handle $H_k$ provided  it does not intersect the curve $a_k$
depicted Figure~\ref{fig:standard_handle_body}.

\begin{proposition}[Lickorish: Lemma-3~\cite{Li}] 
\label{prop:lickorish}
Let $p$ be any simple closed curve on $\Sigma_g$.  There exists
a diffeomorphism $\phi: \Sigma_g \rightarrow \Sigma_g$ such that $\phi (p)$ 
does not meet any handle of $\Sigma_g.$  
\end{proposition}

\section{Lefschetz fibration embedding}\label{sec:lef_fib_embedding}

Recall from Definition~\ref{def:slf}  that a Lefschetz fibration  $(M,\pi:M \rightarrow \Sigma),$
where $\Sigma$ is either a disk or $\mathbb{C}P^1,$ is a SLF, provided that critical values are isolated 
and fibers are connected.   In this
section, we show that  there exists an
embedding of any SLF into certain manifolds of the type $(N^4 \times\mathbb{C}P^1,\pi_2),$  which is fiber preserving in the sense defined in Definition~\ref{def:lef_fib_embedding}.
This result [Theorem~\ref{thm:Lef_fib_embedding}], can be regarded as  the first step towards establishing Theorem~\ref{thm:embedding_4-manifolds}.

\subsection{Flexible embedding in standard position}
\mbox{}

Let us  begin this sub-section by reviewing the  notion of \emph{flexible embedding}.

\begin{definition}[Flexible embedding]
Let $M$ be an orientable closed smooth manifold. 
A smooth embedding $\phi:\Sigma_g \hookrightarrow M$ of a closed orientable
surface $\Sigma_g$ is said to be flexible provided for every 
$f \in \mathcal{M}CG(\Sigma_g)$ there exists a diffeomorphism $\psi$ of $M$ 
isotopic to the identity which maps $\Sigma_g$ to itself and satisfies 
$\phi^{-1} \circ \psi \circ \phi = f.$
\end{definition}

Next, we state a lemma regarding a flexible embedding of any surface of
genus $g$ into a $4$--manifold $N$, which admits a separable  Hopf link. In order to state this lemma, we need to introduce
the following definitions:

\begin{definition} [Embedding in standard position]
An embedding $\phi: \Sigma_g \hookrightarrow N$ of a surface $\Sigma_g$  is
said to be in a standard position provided the following properties
are satisfied:

\begin{enumerate}
\item Every simple closed curve $\gamma$ on $\phi(\Sigma)$ is a boundary of a $2$--disk 
$\mathbb D^2$  intersecting $\phi(\Sigma_g)$ only in $\gamma.$ 
%Furthermore, the disk $\mathbb D$ can 
%be chosen such that it is disjoint form a given $\C P^1 \subset \C P^2$ that
%satisfies the property that $\C P^1 . \C P^1 = +1.$

\item There exists a tubular neighborhood $\mathcal{N} \left(\mathbb D\right)$ of the disk $\mathbb{D}^2$ having the boundary  $\gamma$ such that 
$\mathcal{N} \left(\mathbb D\right)$ is the image of a  coordinate chart  
$\phi_{\gamma}: \C^2 \rightarrow \mathcal{N} \left( \mathbb D \right)$ satisfying 
the following:
 
$\phi_{\gamma}^{-1} (\phi(\Sigma_g) \cap \mathcal{N} \left( \mathbb D \right))$ is  $g^{-1}(1),$ where $g: \C^2 \rightarrow \mathbb{C}$ is the polynomial  map 
$g(z_1, z_2) = z_1.z_2.$
\end{enumerate}
\label{def:stand_position}
\end{definition}

\begin{definition}[Separable Hopf link] \label{def:separable_Hopf_link}We say that a link $l_1\sqcup l_2$ in  a $4$--manifold $N$ is a separable Hopf link provided following properties are satisfied:
\begin{enumerate}
\item There exist an embedding of a $4$--ball $\mathbb{D}^4=\mathbb{D}^2\times \mathbb{D}^2$ in $N$ such that $\partial\mathbb{D}^2\times\lbrace 0 \rbrace \sqcup\lbrace 0 \rbrace \times \partial \mathbb{D}^2=l_1\sqcup l_2$.
\item There exists two disjoint properly embedded discs $\mathcal{D}_1$ and $\mathcal{D}_2$ in $N\setminus (\mathbb{D}^2\times\mathbb{D}^2)^\circ$ such that $\partial\mathcal{D}_1=l_1$ and $\partial\mathcal{D}_2=l_2$.
\end{enumerate}
\end{definition}

\begin{lemma}\label{lem:embedding_stand_flexible}
 Let $N$ be a $4$--manifold which admits a separable Hopf link. Then there exists an embedding $\phi$ of any closed orientable surface $\Sigma_g$
of genus $g$ in $N$ which satisfies the following:

\begin{enumerate}
\item The embedding is flexible.
\item The embedding is in  a standard position.  
%\item  $\phi(\Sigma_g)$ is disjoint form a  given $\mathbb{C}P^1$ embedded in $\mathbb{C}P^2$ which satisfies the property that $\mathbb{C}P^1 . \mathbb{C}P^1 = +1.$

\end{enumerate}
\end{lemma}

Before, we establish this lemma, we would like to point out that the flexible embedding
of $\Sigma_g$ in $N$ was first provided by S. Hirose and  A. Yasuhara 
in \cite{HY}. Our main observation is  that we can achieve the additional property 
of the embedding being in  a standard position, provided that we use Proposition~\ref{prop:lickorish} established by Lickorish in \cite{Li} in conjunction with the techniques from \cite{HY}.

\begin{proof}[Proof of Lemma~\ref{lem:embedding_stand_flexible}]

We want to construct  an embedding of $\Sigma_g$ which is both flexible and in a
standard position. Let $l_1\sqcup l_2$ be a separable Hopf link in $N$. Therefore there exists an embedded $4$--ball $\mathbb{D}^4=\mathbb{D}^2\times\mathbb{D}^2$ in $M$ such that $\partial\mathbb{D}^2\times\lbrace 0 \rbrace \sqcup\lbrace 0 \rbrace \times \partial \mathbb{D}^2=l_1\sqcup l_2$ and there exists two disjoint properly embedded discs $\mathcal{D}_1$ and $\mathcal{D}_2$ in $N\setminus (\mathbb{D}^2\times\mathbb{D}^2)^\circ$ such that $\partial\mathcal{D}_1=l_1$ and $\partial\mathcal{D}_2=l_2$. We regard a $4$--ball $\mathbb{D}^4$ as the $4$--ball $B^4(0,2)$ of radius $2$ in $\mathbb{C}^2$ with its center at the origin. We will also regard $\mathbb{S}^3 \times [1,2]$ as the collar 
$B^4(0,2) \setminus B^4(0,1)$ contained in $N$.

Next,  Observe that the link $l_1\times\lbrace \frac{3}{2}\rbrace \sqcup l_2\times\lbrace \frac{3}{2}\rbrace$ bounds a Hopf band say $\mathcal{H}$ in $\mathbb{S}^3\times \lbrace \frac{3}{2}\rbrace $.  We embed a genus $g$ surface $\Sigma_g$ in $\mathbb{S}^3\times \{\frac{3}{2}\}
\subset \mathbb{S}^3 \times [1,2] \subset N$ as the boundary of standard genus $g$ handle body $H_g$ and disjoint form $\mathcal{H}$ as depicted in Figure~\ref{fig:standard_handle_body}. Then we take ambient connected sum of embedded $\Sigma_g$ and $\mathcal {H}$ in $\mathbb{S}^3\times\lbrace \frac{3}{2}\rbrace$ to 
obtain a surface $\widehat{\Sigma_g}$ with two boundary components  as shown in 
Figure~\ref{fig:flexible_standard_embedding}. Thus by adding two cylinders $l_1 \sqcup l_2 \times [\frac{3}{2},2]$ and two disjoint disc $\mathcal{D}_1$,$\mathcal{D}_2$ to $\widehat{\Sigma_g}$, we obtain an embedding of genus $g$ surface. Let us denote this embedding -- after smoothing the corners -- by $\phi$.   
For a pictorial  description of the embedding $\phi$ we refer the 
reader to Figure~\ref{fig:flexible_standard_embedding}.  We claim that  
the embedding $\phi: \Sigma_g \hookrightarrow N$ is both flexible and in standard position. Let us now establish this claim.

The claim that the embedding is flexible is already established 
in \cite[Theorem: 3.1]{HY}. Let us briefly review the argument. 
First of all, notice that every Lickorish generator $\gamma$ of $\Sigma_g$ embedded in 
$N$ via $\phi$  has -- up to an isotopy -- a Hopf annulus neighborhood 
which is contained in $\mathbb{S}^3 \times \{ \frac{3}{2}\} \subset N$. Next,
recall that  the mapping class group of $\Sigma_g$ is generated by Dehn twists along Lickorish generators, and  in $\mathbb{S}^3$ there exists a
diffeomorphism isotopic to the identity which induces a Dehn twist on a given
Hopf annulus fixing its boundary point wise. In the proof of \cite[Theorem: 3.1]{HY}
it is shown that  this implies that there exists 
a diffeomorphism of $N$ isotopic to the identity which induces a Dehn
twist along a Lickorish generator of $\phi(\Sigma_g).$ The claim 
now follows by successive application of ambient isotopies of $N$ inducing a Dehn twists on Lickorish generators. See also \cite{PPS} for the necessary details.

Let us now show that the embedding is in  a standard position. First of all notice that, by very construction, any simple closed curve on $\phi(\Sigma_g)$ can be isotoped on the surface $\phi(\Sigma_g)$ so that it is contained in $\phi(\Sigma) \cap \mathbb{S}^3 \times \{\frac{3}{2}\}.$  We claim that  any Lickorish generator of 
$\phi(\Sigma_g)$  as well as any curve which does not meet handles\footnote{ Recall that, we say that a simple closed curve $p$ does not meet the handle $H_k$ provided  it does not intersect the curve $a_k$ depicted Figure~\ref{fig:standard_handle_body}.} of
$\phi(\Sigma_g)$ satisfy both the properties necessary for an
embedding to be in a standard position. This is because:

\begin{enumerate}
\item  All  curves mentioned in the claim are unknots in $\mathbb{S}^3 \times \{\frac{3}{2}\}$ hence they bound a disk in $\mathbb{S}^3 \times [1,\frac{3}{2}],$ that meets $\phi(\Sigma)$ only in the given curve.

\item Any curve $\gamma$ mentioned in the claim admits a neighborhood $\mathcal{N}(C)$ in  $\phi(\Sigma_g)$ which is a Hopf band in $\mathbb{S}^3 \times \{\frac{3}{2}\}.$
\end{enumerate}

It follows from both the properties listed above that any  curve $C,$ which
is either a Lickorish generator or is not meeting  any handle, satisfies both the properties necessary for a surface to be in the standard position.

Now, according to Proposition~\ref{prop:lickorish}, given any curve $C$,
there exists a diffeomorphism of $\phi(\Sigma_g)$ which send $C$ to a curve
which  does not meet any handle. Since 
the embedding $\phi$ of $\Sigma_g$ is flexible in $N,$ given a curve $c$ which not a Lickorish 
generator and meets some handles can be isotoped so that now it does not meet 
any handle.  Hence, the
claim that the embedding is also in a standard position follows. 

%
%Observe that till now we have established the properties $(1)$ and $(2).$ Hence,
%in order to establish the lemma we just need to show that the embedding $\phi$
%of $\Sigma_g$ can be made disjoint from a given $\mathbb{C}P^1$ which satisfies
%the property that $\C P^1 . \C P^1 = +1.$
%
%Given $\C P^1$ as in the hypothesis, we know that we can produce a handle decomposition
%of $\mathbb{C}P^2$ starting from the zero handle $B^4(0,2)$ as described 
%while  establishing the properties $(1)$ and $(2)$ earlier,  in such a way that $\C P^1$ is the union of the core of $1$--handle  $H_1$ together with the unknot
%$U$ times $[1,2]$ along which the $1$--handle is attached and the obvious disk
%that the unknot $U \times \{1\}$ bounds in $\mathbb{S}^3 \times \{1\}.$ Now,
%as we observe that the Hopf annulus -- say $\mathcal{H}$ -- $l_i \times \{\frac{3}{2}\}, i = 1, 2$ bounds  can be chosen such that it does not intersect the 
%$U \times \frac{3}{2}$ in $\mathbb{S}^3 \times \frac{3}{2}.$ See the proof of \cite[Theorem: 2]{PPS} for more details. Next, observe that we can ensure that the whole  surface minus the two disk that 
%$l_1 \times \{\frac{3}{2}\}$ and $l_2 \times \{\frac{3}{2}\}$ is contained in arbitrarily small neighborhood of the Hopf annulus 
%$\mathcal{H}$, the third claim follows. Hence the lemma.

\end{proof}

In what follows we will work with embeddings of surfaces in $N$ constructed using the 
procedure described in the proof of Lemma~\ref{lem:embedding_stand_flexible}. We will
used the term \emph{standard embedding}  for any such embedding. More precisely, we have the following:

\begin{definition}[Standard embedding] Let $N$ be a manifold admitting a separable Hopf link.
An embedding $\psi$ of a closed orientable surface $\Sigma_g$ which is isotopic to 
an embedding obtained  following the 
procedure describe in the proof of Lemma~\ref{lem:embedding_stand_flexible} will be called
a standard embedding of $\Sigma_g$ 
\end{definition}

We end this sub-section by establishing an embedding result regarding embeddings of
mapping tori in $N \times \mathbb{S}^1.$ Recall that given a surface $\Sigma,$ the mapping
torus of $\Sigma$ with monodromy $g$, with $g \in \mathcal{M}CG(\Sigma),$ is the quotient space  $\Sigma \times [0,1] /~ ,$ where $(x,0) ~ (g(x),1).$ We will denote the mapping torus by $\mathcal{M}T(\Sigma, g).$ $\mathcal{M}T(\Sigma, g)$ is a fiber bundle over $\mathbb{S}^1.$ Our next
lemma establishes a fiber preserving embedding any mapping tours of $\Sigma$ in
$N \times \mathbb{S}^1.$ More precisely,

\begin{lemma}\label{lem:embedding_mapping_torus}
Let $N$ be a manifold admitting a separable Hopf lank and let $\phi: \Sigma \hookrightarrow N$ be
a standard embedding of $\Sigma.$  Let $g$ be an element of the mapping class group of $\Sigma$
Let $M$ be the mapping tours of $\Sigma$ with monodromy $g.$ There exists an embedding
of $M$ in $N \times \mathbb{S}^1$ which is fiber preserving. 
\end{lemma}

\begin{proof}
Since the embedding is standard, for  $ g \in \mathcal{M}CG(\Sigma)$ there
exists a family $f_t$ of diffeomorphisms of $\N$ with $f_0 = Id$ and $f_1$ restricted
to $\Sigma$ satisfies $\phi^{-1} \circ f_1 \circ \phi = g. $ This implies $\mathcal{M}T(\Sigma,g)$ is
contained in $\mathcal{M}T(N, \phi_1).$ Since $f_1$ is isotopic to the identity, 
$\mathcal{M}T(N, f_1) = N \times \mathbb{S}^1.$ Hence the lemma.
\end{proof}

Before we proceed. We would like to point out that Lemma~\ref{lem:embedding_mapping_torus} was
implicitly established in~\cite{PPS}.

\begin{figure}
\centering
\includegraphics[height=6cm,width=8cm ]{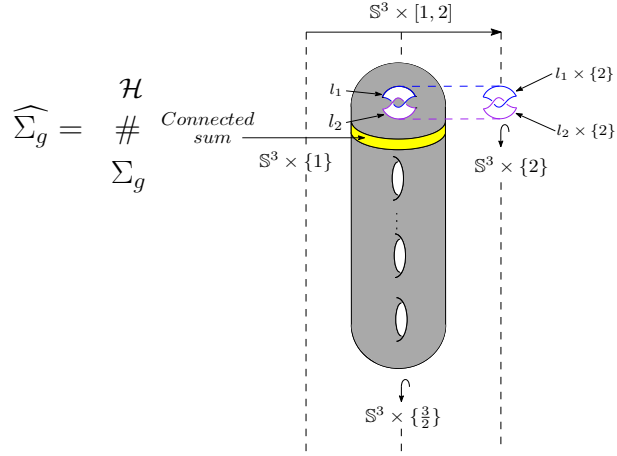}
\caption{The figure depicts the embedding of the surface $\Sigma_g$ which is 
flexible as well as in the standard position. Figure depicts the collar
$\mathbb{S}^3 \times [1,2] \subset N$ with dashed lines representing
$\mathbb{S}^3$ at levels $1$, $2$ and $\frac{3}{2}.$  
%Notice that the surface is
%embedded disjoint from the zero section which in the figure consist of 
%the red dashed  cylinder $l \times [1,2]$ with two caps. The disk $\mathcal{D}_1$ --
%depicted in the figure in green color -- contained in 
%$\mathbb{S}^3 \times \{1\}$ and the core of $1$--handle attached along 
%$U \times \{2\} $.
}
\label{fig:flexible_standard_embedding}
\end{figure}

\subsection{The existence of Lefschetz fibration embedding}\mbox{}

We are now in a position to state and prove our main result regarding 
\emph{Lefschetz fibration embeddings}. 

%
%First of all recall that 
%$\C P^3$ admits a Lefschetz pencil. That is, there exists a map
%  $\pi_{\C P^3}: \C P^3 \setminus \C P^2 \rightarrow \C P^1$ which is 
%  a trivial fiber bundle with fibers $\C P^2 \setminus \C P^1.$ 
%  Since the fibration associated to the map $\pi_{\C P^3}$ is trivial,
 As usual, we denote the map  $N \times \C P^1$ to $ \C P^1$ 
  corresponding to the projection on the second factor by $\pi_2.$ 

%The reason for this notation is that $\mathbb{C}P^2\times\mathbb{C}P^1$ is a blow-up of $\mathbb{C}P^3$ along 
%a $\mathbb{C}P^1.$  

\begin{definition} [Lefschetz fibration embedding] \label{def:lef_fib_embedding}
Let $(M,\pi:M \rightarrow \Sigma)$  be a Lefschetz fibration, where $\Sigma$ is 
$2$--disk or 
$\mathbb{C}P^1$.  An embedding $f: M \to N\times\mathbb{C}P^1$ of a manifold $M$ into a manifold $N \times\mathbb{C}P^1$ is said to be a \emph{Lefschetz fibration embedding} provided
$\pi_{2}\circ f$=$  i \circ \pi,$ where $i$ is  an inclusion of $\mathbb{D}^2$
in $\mathbb{C} P^1$ when $\partial M \neq \emptyset,$ otherwise it is  the identity.
\end{definition}

\begin{figure}
\centering
\includegraphics[scale=0.6]{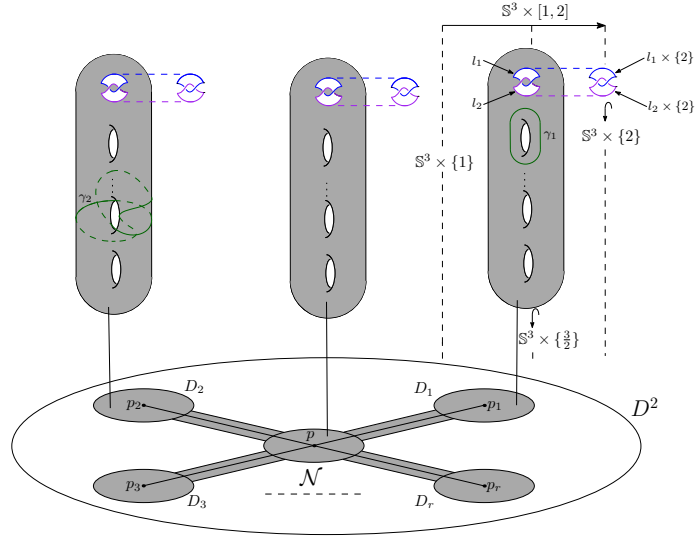}
\caption{The figure depicts part of a Lefschetz fibration $(M, \pi)$ over a disk embedded as Lefschetz fibration  in
the (Lefschetz) fibration 
$ N \times D^2 \rightarrow D^2.$ The embedding is such that
the generic fiber of $(M, \pi)$ is a flexible embedding in the standard position in
$N.$  The curves on the surface depicts the vanishing
cycles  $\gamma_i$'s.}
\label{fig:Lefschetz_fibration}
\end{figure}

%\begin{remark}
%Notice that if a closed Lefschetz fibration $(M, \pi)$ admits a Lefschetz fibration
%embedding $f$  in $\mathbb{C}P^3$ then $f(M) \cap \mathbb{C}P^2 = \emptyset.$
%That is, the Lefschetz fibration $(M, \pi)$ admits a Lefschetz fibration embedding
%in the trivial fiber bundle $\pi_{\C P^3}:\C P^3 \setminus \C P^2 \rightarrow \C P^1.$
%\end{remark}

\begin{theorem} 
\label{thm:Lef_fib_embedding}
 Let $M$ be an orientable smooth $4$--manifold. Let $N$ be a $4$--manifold which admits a separable Hopf link. 
If $\pi:M \rightarrow \Sigma,$
where $\Sigma$ is either $\mathbb{C}P^1$ or a $2$--disk $\mathbb{D}^2$,
is a simplified Lefschetz fibration (SLF) of $M$ having genus  $g$ fibers with $g \geq 1$, then there exists a Lefschetz fibration 
embedding of $(M, \pi)$ in $(N\times\mathbb{C}P^1, \pi_2).$ 
\end{theorem}

\begin{proof} 
Let us first  provide a proof of the theorem, when  $\Sigma= \C P^1.$ In this
case, $M$ is a closed orientable manifold admitting a SLF $\pi: M \rightarrow \C P^1.$

Let $c_1, c_2,  \cdots c_k$ be $k$ critical points of the Lefschetz fibration
$(M, \pi).$ Since the Lefschetz fibration is simple, 
$\pi(c_1) =p_1, \pi(c_2) =p_2, \cdots,$ and $ \pi(c_k)= p_k$ are distinct points on 
$\C P^1.$ Also, recall that   that the genus $g$ of the generic fiber is bigger than or equal to $2.$ Let $\gamma_i$ be the vanishing cycle corresponding to the critical point $c_i$ on a generic fiber $\Sigma_g$ of the SLF.  

Let $U_i$ be the open ball in $M$ around $c_i$ such that on $U_i$ we have co-ordinates $(z_1, z_2)$ such that  $\pi$ in this co-ordinates is given by $(z_1, z_2) \rightarrow z_1.z_2.$ Let 
$\widetilde{D}_i = \pi(U_i) \subset \C P^1.$ Let $D_i$ be an open disk containing
$p_i$  with  $\overline{D_i} \subset \widetilde{D}_i.$

First of all consider an embedding $\phi$ of the fiber  $\Sigma_g$ in $N$ 
which is a standard embedding. Recall that the existence of such
an embedding is the content of Lemma~\ref{lem:embedding_stand_flexible}.

Using the flexibility of the embedding $\phi,$ we first produce an embedding 
$\widehat{f}$ of
$M \setminus \sqcup_{i=1}^{k} \pi^{-1}(D_i)$ in the manifold 
$ N\ 
\times \left( \C P^1 \setminus \sqcup_{i=1}^k D_i\right)$ such that the following diagram commutes:

\begin{equation}
\begin{tikzcd}
 M \setminus \sqcup_{i=1}^{k} \pi^{-1}(D_i) \arrow[r, "\widehat{f}"] \arrow[d,"\pi" ] &  N\ 
\times \left( \C P^1 \setminus \sqcup_{i=1}^k D_i\right) \arrow[d,"\pi_2"]  \\
\C P^1 \setminus \sqcup_{i=1}^{k}D_i \arrow[r, "Id"] & \C P^1 \setminus \sqcup_{i=1}^{k}D_i .
\end{tikzcd}
\end{equation}

In order to do this, we observe that  the embedding of $\Sigma_g$ in $N$ is standard. Hence
Lemma~\ref{lem:embedding_mapping_torus} implies given an element $\psi \in \mathcal{M}CG(\Sigma_g)$ there exists an embedding $\Psi$ of the mapping torus, $\mathcal{M}T (\Sigma_g, \psi)$,   in the trivial fiber bundle  $\pi_2:  M  \times \mathbb{S}^1 \rightarrow \mathbb{S}^1$  such that
the following diagram commutes:

\begin{equation}
\begin{tikzcd}
 \mathcal{M}T(\Sigma_g, \psi) \arrow[r, "\Psi" ] \arrow[d, "\pi"] & N\times
 \mathbb{S}^1 \arrow[d]  \\
\mathbb{S}^1 \arrow[r, "Id"] & \mathbb{S}^1 .
\end{tikzcd}
\end{equation}

Next, considering   $\partial D_i \subset \mathbb{C} P^1 = \mathbb{S}^1$ then
it follows from the existence of an embedding $\phi$ satisfying the diagram $(2)$ that there is an embedding of the mapping torus 
$\mathcal{M}T(\Sigma_g, \tau_{\gamma_i})$ in $N  \times \partial D_i,$ where $\tau_{\gamma_i}$ denotes the Dehn twist along the curve
$\gamma_i.$  Now take arcs
connecting a point on $\partial D_i$ to a fixed regular point $p$ for
the map $\pi$  in $\C P^1$ as  depicted in Figure~\ref{fig:Lefschetz_fibration}.

Since the Lefschetz fibration $(M, \pi)$ restricted to  a regular 
neighborhood $\mathcal{N}$ of $D_i$'s together with arcs connecting them satisfies the
following:

\begin{enumerate}
\item   $\pi^{-1} (\partial D_i)$ is the mapping torus 
$\mathcal{M}T(\Sigma_g, \tau_{\gamma_i}),$

\item  $M$ restricted to $\partial \mathcal{N}$ is the manifold $\Sigma_g \times 
\mathbb{S}^1$ as $\displaystyle\prod_{i=1}^{k} \tau_{\gamma_i} = Id$ in $\mathcal{M}CG(\Sigma_g),$

\item the complement of $\mathcal{N}$ is a disk in $\C P^1,$

\item and the genus $g \geq 2,$

\end{enumerate}
\noindent we get the required embedding $\widehat{f}$ such that the diagram
$(1)$ commutes.

%We refer to Figure~\ref{fig:Lefschetz_fibration} for a pictorial description 
%the embedding $\widehat{f}$.  

Our next step is to  show how to extend this embedding to produce a Lefschetz fibration embedding of $f$ of $M$ in $N\times\mathbb{C}P^1.$  For this  the property that the embedding $\phi$ of $\Sigma_g$ is also in the standard position  is 
required.

Since the embedding $\phi$ is in a standard position -- by
the definition of an embedding in a standard position given in \ref{def:stand_position} -- there exists an embedding of $\phi_{\gamma_i}:\C^2 \hookrightarrow  N$ which satisfies the second property 
listed in  Definition~\ref{def:stand_position}.

Next, for each critical point $c_i$, we claim that, we have following commutative diagram: 

\begin{equation}\label{diag:commutative_local_maps}
 \begin{tikzcd}
 U_i\subset M\arrow{r}{\phi_i}\arrow{d}{\pi}
 &\C^2\arrow{r}{i}\arrow{d}{g}
 &\C^2\times\C \arrow{r}{f_{c_i}}\arrow{d}{P}
 &N\times \C P^1\arrow{d}{\pi_2}\\
 \widetilde{D}_i\arrow{r}{\phi}&\C\arrow{r}{Id}&\C \arrow{r}{\phi^{-1}}&\widetilde{D}_i ,\end{tikzcd}
 \end{equation}
\noindent where the definitions of the maps appearing in the diagram are as follows:
\begin{enumerate}
\item $\phi_i:U_i\subset M \to \C^2$ and $\phi: \widetilde{D}_i\subset \C P^1 \to \C$ are orientation preserving parameterizations around critical point $c_i$ of $\pi$ and $\pi(c_i)$ respectively such that left square commutes in the diagram above,
\item $i:\C ^2 \to \C^2 \times \C$ and $g:\C^2 \to \C$ are defined as
 $i(z_1,z_2)=(z_1,z_2,0)$ and $g(z_1,z_2)=z_1.z_2$,
\item $f_{c_i}:\C^2\times\C \to N \times \C P^1$ and $P:\C^2\times\C \to \C$ are defined as
 
\noindent $f_{c_i} (z_1,z_2,z_3)=(\phi_{\gamma_i}(z_1,z_2),\phi^{-1}(z_1.z_2+z_3))$ and $P(z_1,z_2,z_3)=z_1.z_2+z_3$.
\end{enumerate}
\noindent The commutativity of the middle square is follows directly from definitions of  maps $g,i$ and $P$. Also the commutativity of the last square is clear by the definition of the map  $f_{c_i}$. Next, we observe that the commutative diagram~\ref{diag:commutative_local_maps} allows us to extend the embedding $\widehat{f}$ to the embedding $\widehat{f}_{c_i}$ of  $M \setminus \sqcup_{i=1}^{k} \pi^{-1}(D_i) \cup U_i$. This is possible because $\widehat{f}$ and $f_{c_i}\circ i\circ\phi_i$ agree on the overlapping region of the domain. Hence, $\widehat{f}$ and $f_{c_i}\circ i\circ\phi_i$ together defines  a map $\widehat{f}_{c_i}$.

Let us now notice that this allows us to extend the embedding $\widehat{f}_{c_i}$ 
 to an embedding $\widehat{f}_{c_i}$ of 
 $ W_{c_i} =M \setminus \left( \bigcup\limits_{l=1}^{i-1}\pi^{-1}(D_l) \bigcup\limits_{l=i+1}^{k}  \pi^{-1} (D_{l})\right)$ 
in $N \times \C P^1$ such that the following diagram commutes:

\begin{equation}
\begin{tikzcd}
  W_{c_i} \arrow[r,"\widehat{f}_{c_i}"] \arrow[d, "\pi"] & \widehat{f}_{c_i}(W_{c_i}) \subset N \times \C P^1 \arrow[d, "\pi_2"] \\
 \pi(W_{c_i}) \subset \C P^1 \arrow[r, "Id"] & \pi_2(\widehat{f}_{c_i}) = \pi(W_{c_i}).   
\end{tikzcd}
\end{equation}

Observe that by construction the embeddings $\widehat{f}_{c_i}$ and 
$\widehat{f}_{c_j}$ agree on on $W_{c_i} \cap W_{c_j}.$  Since 
$M = \cup_{i=1}^k  W_{c_i}$ we get an embedding $f$ of $M$ with required 
properties. This completes our argument in case when $\Sigma= \C P^1$. 

The case, when $\Sigma= \mathbb D^2$ the argument is essentially same.  The only
difference is that the product
 $\displaystyle\prod_{i=1}^{k} \tau_{\gamma_i}$ need not be the identity.  
However, notice that since  $\Sigma = \mathbb{D}^2$ the same argument produces
an embedding such that the monodromy along $\partial \mathbb{D}^2$ is precisely 
 $\displaystyle\prod_{i=1}^{k} \tau_{\gamma_i}.$

\end{proof}

\section{Embedding of orientable $4$-manifolds via SBLF}
\label{sec:embed_in_product}

The purpose of this section is to establish a class of $6$--manifolds in which all closed smooth orientable
$4$--manifolds embed. As mentioned earlier,
we will use the SBLF decomposition of a closed orientable smooth $4$--manifold 
for constructing embeddings. We first need the following:

\begin{definition}[$1$--fold simple singular fibration] Let $(M, \partial M)$ be an orientable smooth $4$--manifold with boundary and 
let $f: M \rightarrow  [-1,1] \times \mathbb{S}^1$ be a smooth surjective  map which satisfies the following:

\begin{enumerate}
\item  There exists a unique embedded circle  $Z_f$ in $M$ of $1$-fold singularity for $f$ such that
$f (Z_f)$ is an embedded  circle in $[-1,1] \times \mathbb{S}^1$ which
is  ambiently isotopic to the circle  $ \{ 0 \} \times \mathbb{S}^1.$

\item every $x \in M \setminus Z_f$ is a regular value for the map $f$

\item $\partial M = f^{-1}\left(\{-1\} \times \mathbb{S}^1 \sqcup f^{-1}\{1\} \times \mathbb{S}^1\right).$

\end{enumerate}

Then, we say that $f: M\rightarrow [-1,1] \times \mathbb{S}^1$ is a 
\emph{$1$--fold simple singular fibration}.
\end{definition}

\begin{remark}\mbox{}

\begin{enumerate}[label=(\alph*)]
\item Since $f:M \rightarrow [-1,1] \times \mathbb{S}^1$ has a unique embedded
singular locus $Z_f$ which projects to a circle $C$ isotopic to 
$\{0\} \times \mathbb{S}^1$ the inverse image of any regular value is a closed
surface $\Sigma$ whose genus is either $g$  or $g+1$ for some $g \in \mathbb{N} \cup \{0\}.$ We call a fiber with genus $k$ as a  lower genus fiber.

\item Observe that as we cross the $f(Z_f)$  a round $1$--handle is added to
a manifold  diffeomorphic to $\Sigma_g \times A,$ where $A$ is an annulus.
   
\item We will always use the convention that fiber over $\{-1\} \times \mathbb{S}^1$
have lower genus.

\end{enumerate}
\end{remark}

\begin{lemma}Let $(M,\partial M)$ be an orientable smooth $4$--manifold with boundary and $f: M\rightarrow [-1,1]\times\mathbb{S}^1$ be a $1$--fold simple singular fibration. Let $N$ be a $4$--manifold which admits a separable Hopf link. Then there exists an embedding $\psi:M\rightarrow N\times [-1,1]\times \mathbb{S}^1$ such that following properties are satisfied:

\begin{enumerate}
 \item The following diagram commutes:

\begin{equation}
\begin{tikzcd}
 M \arrow[r, "\psi"] \arrow[d, "f"] &  N \times [-1,1] \times \mathbb{S}^1  \arrow[d, "\pi_2"]  \\
\left[-1,1\right] \times \mathbb{S}^1 \arrow[r,  "Id" ] & \left[-1,1\right] \times \mathbb{S}^1. 
\end{tikzcd}
\end{equation}
\item  Given a standard embedding    $\phi$ of a surface of genus $g+1$ in $N$, we can ensure that $\psi$ restricted to any higher genus fiber send the fiber to a surface in $N$ which is isotopic to the given
embedding $\phi.$

\end{enumerate}
  
\label{lem:simple_singular_fibration_embedding}
\end{lemma}

\begin{proof}

Let us denote by $M_0 = f^{-1}(\{-1\} \times \mathbb{S}^1),$ and 
$M_1 = f^{-1}(\{1\} \times \mathbb{S}^1.$
We know that $\partial M = M_0 \sqcup M_1.$
Observe that $M_1$ is a mapping torus  over $\mathbb{S}^1$ with fiber $\Sigma_{g+1}.$
Recall that any mapping torus over $\mathbb{S}^1$ is determined by  its
monodromy -- an element of $\mathcal{M}CG(\Sigma_g).$ 
Let $\phi$ be the monodromy for the fiber bundle $M_1$ over 
$\mathbb{S}^1.$ Further, since $f:(M, \partial M) \rightarrow [-1,1] \times \mathbb{S}^1$ is 
a $1$--fold simple singular fibration, we have the following: there exists a curve $c$ in $\Sigma_{g+1}$ which is mapped to itself by $\phi$~\cite[p. 10895]{BO}, and  the boundary component $M_0$ is obtained from $M_1$ by the following procedure:

First cut $\Sigma_{g+1}$ along $c,$ and  attach
to the resulting surface  a pair of disks -- say $D_1$ and $D_2$. Now form the mapping torus of the resulting surface $\Sigma_g$ with  monodromy the map $\phi$ restricted to  $\Sigma_g.$ 

This also implies that we can obtain $(M, \partial M)$  by suitably adding a round $1$--handle 
along a pair of points times $\mathbb{S}^1$ such that each disk $D_i \times \mathbb{S}^1$
 contains a circle of the round attaching sphere.

 Now, let $i: \Sigma_{g+1} \subset  N$ be a  standard embedding of $\Sigma_{g+1}$
in $N$  Since the embedding is standard, we know every simple close  curve $\gamma$ on 
$\Sigma_{g+1}$ bounds a disk in $D$ in $N$ such that the intersection of this disk with $N$ is 
$\gamma.$  Furthermore, recall that any simple closed curve in a standard embedding of 
$\Sigma_{g+1}$ can be assumed to be disjoint from the separable Hopf link, and the pair of disjoint disks that the link bounds. This implies  that there exist a $4$--ball $B$ containing the disk $D$ such that $\Sigma_{g+1} \cap B^4$ is an annulus  $A$ and $\partial A$ is a pair of unlinked unknots in $\partial B^4.$

Since the embedding is standard, from Lemma~\ref{lem:embedding_mapping_torus} 
it follows  that there exist
a fiber preserving  embedding of $M_1$ in $N \times \{1\} \times \mathbb{S}^1.$ Since $\phi$ send $c$ to itself $\phi(c) = \pm c.$  Since the curve $c$ bounds disk in $\Sigma_g,$ without
loss of generality we can assume that $\phi(c) = c.$ 

We know that the embedding of a surface  $\Sigma_g$ obtained by
cutting $\Sigma_{g+1}$ along the curve $c$ agrees with $\Sigma_{g+1}$ everywhere
except in the ball $B^4.$ Since the ball $B^4$ is disjoint form the separable Hopf link and 
the pair of disjoint disks that the link bounds, we get that the embedding of  $\Sigma_g$ given by
by cutting $\Sigma_g$ is also standard. Hence, applying Lemma~\ref{lem:embedding_mapping_torus},
we get an embedding of $M_0$
in $N \times \{-1\} \times \mathbb{S}^1$ which is also fiber preserving. 

Observe that by very construction the embedding of $\partial M = M_0 \sqcup M_1$ can be extended 
to an embedding  $\widehat{\psi}$ of $(M, \partial M) \setminus  \mathcal{N}$ in 
$N \setminus B^4 \times [-1,1] \times \mathbb{S}^1,$ where $N$ is a neighborhood of 
$1$--fold singularity. Furthermore, we can assume that the following diagram commutes:

\begin{equation}
\begin{tikzcd}
 M \setminus \mathcal{N} \arrow[r, "\widehat{\psi}"] \arrow[d, "f"] &  N \setminus B^4 \times [-1,1] \times \mathbb{S}^1  \arrow[d, "\pi_2"]  \\
\left[-1,1\right] \times \mathbb{S}^1 \arrow[r,  "Id" ] & \left[-1,1\right] \times \mathbb{S}^1. 
\end{tikzcd}
\end{equation}

 Hence, in order to establish the lemma, we need to extend the embedding constructed so far
 in the region $\mathcal{N}.$  We can assume that $\mathcal{N}$ is a tubular neighborhood of
 the $1$-fold critical locus, and hence can be identified with $B^3 \times \mathbb{S}^1.$

Let $(x,y,z,\theta)$ be co-ordinates on  a tubular neighborhood 
$\mathcal{N} = B^3\times \mathbb{S}^1$ of the singular locus  $Z_f$ of $f$ such that the map $f$ sends $(x,y,z,\theta)$ to $(-x^2+y^2+z^2,\theta).$ 
Let us  embed $B^3\times \mathbb{S}^1$ in $B^4(0,1)\times[-1,1]\times\mathbb{S}^1.$ 
The  embedding  $\widehat{\psi_1}:B^3\times \mathbb{S}^1\rightarrow B^4(0,1)\times[-1,1]\times\mathbb{S}^1$  is defined as $\widehat{\psi_1}(x,y,z,\theta)=(x,y,z,0,-x^2+y^2+z^2,\theta)$. We can see $\widehat{\psi_1}$ is defined such that following diagram commutes: 

\begin{equation}
\begin{tikzcd}
 B^3\times \mathbb{S}^1 \arrow[r, "\widehat{\psi_1}"] \subset M \arrow[d, "f"] &  B^4(0,1) \times [-1,1] \times \mathbb{S}^1 \subset N\times[-1,1]\times\mathbb{S}^1 \arrow[d, "\pi_2"]  \\
\left[-1,1\right] \times \mathbb{S}^1 \arrow[r,  "Id" ] & \left[-1,1\right] \times \mathbb{S}^1. 
\end{tikzcd}
\end{equation}

\begin{figure}
\centering
\includegraphics[scale=.7]{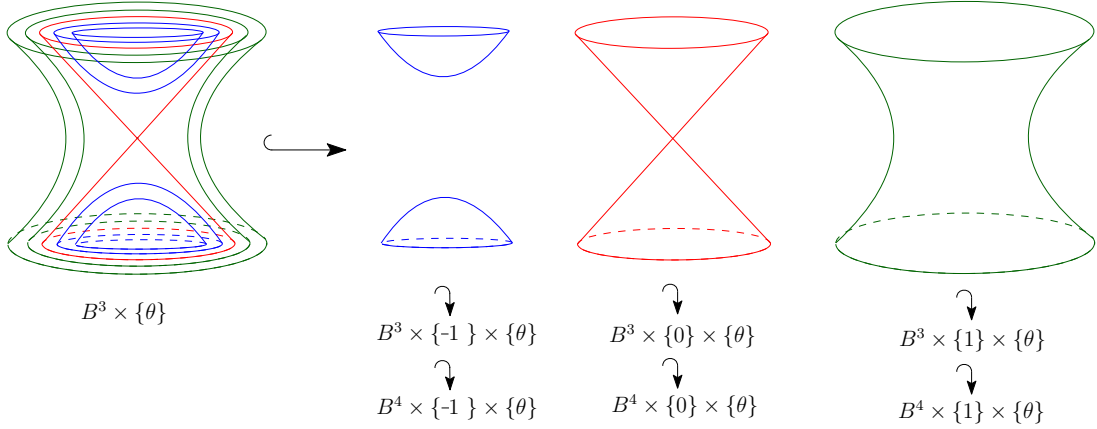}
\caption{Simple Lefschetz fibration embedding}
\label{fig:simple_lefschetz_fibration}
\end{figure}

Observe that the embedding $\widehat{\psi}_1$ has the property that for 
each $\{t\} \times \mathbb{S}^1,$ the intersection of $f^{-1} (\{t\} \times \mathbb{S}^1$ with
$\partial B^4 \times \{t\} \times \mathbb{S}^1$ is a pair of unliked unknot. This implies that
by perturbing the embedding $\widehat{\psi}$ slightly, we can assume that both embeddings
agrees near the boundary to produce an embedding $\psi$ of $M$ in $N \times [-1,1] \times \mathbb{S}^1.$

Clearly, $\psi$ is the required embedding. This shows that we can produce an
embedding of $(M,\partial M)$ in $N$ satisfying the property $(2).$ Since there always
exists a standard  embedding of $\Sigma_{g+1},$ the lemma follows.
\end{proof}

\begin{theorem}\label{thm:embedding_in_CP^2_times_CP^1} Let $M$ be an orientable closed smooth $4$--manifold.  Let $N$ be a $4$--manifold which admits a separable Hopf link. Then there exists an embedding $\psi:M\rightarrow N\times\C P^1.$ 
\end{theorem}

\begin{proof}
Let $M$ be a closed oriented $4$--manifold. By Theorem~\ref{thm:existence_SBLF} there exists a smooth map $f:M\rightarrow\C P^1$ which defines SBLF such that the lower genus fiber $\Sigma_g$ of $f$ has genus bigger than $1$. Therefore, We have a decomposition of $M$, $M=X_1\sqcup X_2\sqcup \Sigma_g\times D_2,$ satisfying the following:
\begin{enumerate}
\item $X_1=f^{-1}(D_1)$ with  where $D_1$ is a disc in $\C P^1$ such that  in the interior of $D_1$ contains all Lefschetz 
critical values of $f.$

\item $f$ restricted to $X_2$ is $1$-fold singular fibration. 

\item $\Sigma_g\times D_2=f^{-1}(D_2)$, where $D_2$ is a disc in $\C P^1$ containing no critical points of $f$ with $\{-1\}\times\mathbb{S}^1=\partial D_2$. 

\item Identifications along boundaries of adjacent regions is always via the identity map.

\end{enumerate}

 It follows from  Theorem~\ref{thm:Lef_fib_embedding} and 
 Lemma~\ref{lem:simple_singular_fibration_embedding} that each piece of $M$
 embeds in $N \times \C P^1.$  Also, it is clear from the second property listed in the statement
 of Lemma~\ref{lem:simple_singular_fibration_embedding}  that embeddings of each piece can
 be arranged such that in the overlapping region they agree. This clearly implies
 that we have an embedding of $M$ in $N \times \C P^1$ as claimed. 
 \end{proof}
 
 \begin{remark}\mbox{}
 \begin{enumerate}[label=(\alph*)]
 \item The embedding $\psi:M\to N\times \mathbb{C}P^1$ produced in Theorem~\ref{thm:embedding_in_CP^2_times_CP^1} satisfies $\psi\circ\pi_2=f$, where $f:M\to\mathbb{C}P^1$ is SBLF associated to $M$ and $\pi_2:N\times\mathbb{C}P^1\to \mathbb{C}P^1$ is projection onto second factor of $N\times \mathbb{C}P^1$. In this case, the embedding $\psi$ is termed as \emph{SBLF embedding}.
 
 \item In general, given a fiber bundle $\pi:X^6 \rightarrow \C P^1$ and an embedding of $M^4$
 in $X^6$ such that $\pi$ restricted to $M$ induces a SBLF on $M$ will also be termed as 
 an SBLF embedding.
 
 \end{enumerate}
 \end{remark}
 
 \section{Embeddings in $\mathbb{R}^7$}

In this section we give a new proof of the fact that every closed smooth orientable $4$--manifold
admits a smooth embedding in $\mathbb{R}^7.$

 \begin{theorem}\label{thm:embedding_in_R^7}
 
  Every $4$--manifold admits a smooth embedding  in $\mathbb{R}^7$.
 \end{theorem}

   \begin{figure}[H]
 \includegraphics[scale=0.7]{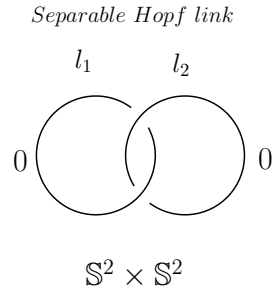}
 \caption{Figure  depicts the kirby diagram of $\mathbb{S}^2\times\mathbb{S}^2.$ Observe that
  attaching circles of $2$--handles  form a Hopf link  in the boundary of the unique $0$--handle, and  they bound disjoint disk corresponding to attaching disks  in $\mathbb{S}^2 \times \mathbb{S}^2$.}
 \label{fig:kirby_diagram_S^2_cross_S^2}
 \end{figure}

\begin{proof} Consider the $4$--manifold $\mathbb{S}^2\times\mathbb{S}^2$. We observe  that 
$\mathbb{S}^2\times\mathbb{S}^2$ admits a separable  Hopf link. This is because $\mathbb{S}^2 \times \mathbb{S}^2$ admits
a handle decomposition consisting of a unique $0$--handle $H_0$ one which a pair of two
$2$--handles are attached such that the attaching circles form a Hopf link in $\partial H_0.$ 
For a pictorial description of this handle decomposition, we refer to  
Figure~\ref{fig:kirby_diagram_S^2_cross_S^2}, where  we have presented 
a Kirby diagram of $\mathbb{S}^2 \times \mathbb{S}^2.$  This clearly implies that the Hopf link consisting of the pair of attaching circles is a separable Hopf 
link.   Thus by Theorem~\ref{thm:embedding_in_CP^2_times_CP^1}, every $4$--manifold embeds in $\mathbb{S}^2\times\mathbb{S}^2\times\mathbb{C}P^1=\mathbb{S}^2\times\mathbb{S}^2\times\mathbb{S}^2.$ Now as $\mathbb{S}^2\times\mathbb{S}^2\times\mathbb{S}^2$ embeds in $\mathbb{R}^7$, proof of the corollary follows.
\end{proof}

\section{Embeddings in $\mathbb{C}P^3$}\label{embed_in_proj}

Let us now establish Theorem~\ref{thm:embedding_4-manifolds}. 
As mentioned in the introduction, the first step of the  proof involves construction of a  specific SBLF on $M \# \C P^2 \# \overline{\C P^2}.$ We then use this SBLF to
produce an embedding of $M \# \C P^2 \# \overline{\C P^2}$ in the blow-up 
$\mathcal{B}_{\mathbb{C}P^1}(\C P^3)$ of $\C P^3$ along $\C P^1.$ hence there is an   embedding  Furthermore, we show that this embedding can be constructed such that when we blow-down 
$\mathcal{B}_{\mathbb{C}P^1}(\C P^3)$, we get an embedding of $M$ in $\C P^3.$
We begin by  reviewing notions related to blow-up and blow-down.  

\subsection{ Generalized Lefschetz pencil} \mbox 

 \begin{definition}[Generalized Lefschetz pencil]\label{def:lefschetz_pencil} Let $M$ be an orientable  smooth $4$--manifold. A \emph{generalized Lefschetz pencil} associated to $M$ is a map $\pi:M\setminus B \to \mathbb{C} P^1$ such that the following properties are satisfied:
 \begin{enumerate}
 \item $B$ is finite.
 \item $\pi:M\setminus B \to \mathbb{C} P^1$ is a Lefschetz fibration.
 \item For every point $b\in B$ there are  parameterizations --
 \emph{not necessarily preserving orientations} -- $\phi: U\subset M \to \mathbb{C}^2$ that satisfies the following:
 
\begin{enumerate}
\item $b\in U$ and $\phi(b)=0\in \mathbb{C}^2$

\item For the map $g:\mathbb{C}^2 \rightarrow \mathbb{C}P^1$ given by $g(z_1,z_2)=\frac{z_2}{z_1}$, the following diagram commutes:
 \begin{equation}
 \begin{tikzcd}
 U  \arrow{r}{\phi}\arrow{d}{\pi}
 & \mathbb{C}^2   \arrow{d}{g}\\
  \mathbb{C}P^1 \arrow{r}{Id} &\mathbb{C}P^1
 \end{tikzcd}.
 \end{equation}
 
 \end{enumerate}
 
 \end{enumerate}
 In this case, we call $B$ as a base locus of a generalized Lefschetz pencil associated to $M$.
 \end{definition}
 
 \begin{remark}\mbox{}
 \begin{enumerate}[label=(\alph*)]
 
\item We would like to emphasis that the notion of generalized Lefschetz pencil defined above is weaker than the notion of  Lefschetz pencil. Generally one demands that  $M$ and $\C P^1$ are
oriented and the parameterizations $\phi: U \subset M \rightarrow \C^2$ and $\psi:V\subset \mathbb{C} P^1\to \mathbb{C}$ are orientation preserving in Definition~\ref{def:lefschetz_pencil}. 
 
 \item If the fibration $\pi:M\setminus B \to \mathbb{C} P^1$ is simplified Lefschetz fibration, the pencil is termed as \emph{generalized simplified broken Lefschetz pencil} or generalized SBLF in short.
% \item In case, we allow the fibration $\pi:M\setminus B \to \mathbb{C} P^1$  to be achiral Lefschetz fibration, the pencil is termed as \emph{generalized achiral Lefschetz pencil}.
% \item It is clear that  we can define \emph{generalized simplified achiral Lefschetz pencil} and \emph{generalized simplified broken achiral Lefschetz pencil} in a similar way (see \cite{BO}). 
 \item If the fibration $\pi:M\setminus B \to \mathbb{C} P^1$ is simplified broken Lefschetz fibration and the parameterizations $\phi: U \subset M \rightarrow \C^2$ and $\psi:V\subset \mathbb{C} P^1\to \mathbb{C}$ are orientation preserving, the pencil is termed as \emph{simplified broken Lefschetz pencil} (SBLP).
\end{enumerate}
 \end{remark}
 
 \subsection{Topological  blow-up and  blow-down of $4$--manifolds}\mbox{}
 
We begin by recalling few standard facts from \cite{GS} about the tautological line bundle over $\mathbb{C}P^1$ and the bundle (complex) dual to this bundle.  
 
 Consider the tautological line bundle $\tau_{\mathbb{C}P}$ over 
 $\mathbb{C}P^1,$ and the bundle $\tau_{\mathbb{C}P^1}^*$ dual to the bundle $\tau_{\mathbb{C}P^1}.$   Let $\mathcal{Z_{\tau}}$ denote the zero section of the bundle $\tau_{\mathbb{C}P^1}, $
 while $\mathcal{Z}_{\tau^{*}} $ denote the zero section of the bundle $\tau_{\mathbb{C}P^1}^*.$
 
 We know that  $\tau_{\mathbb{C}P^1} \setminus \mathcal{Z}_{\tau}, $ and  
  $\tau_{\mathbb{C}P^1}^* \setminus \mathcal{Z}_{\tau^{*}} $  are   diffeomorphic to $\mathbb{R}^4 \setminus \{ 0 \}$ by diffeomorphisms 
 coming from the restrictions of the projection of second factor for the corresponding bundles.  
 We fix this identification of
 the complement of zero sections with $\mathbb{R}^4 \setminus \{0\}$ for both these bundles.

 \begin{definition}[Topological blow-up] Let $M$ a smooth $4$--manifolds. 
 Let $p$ be a point in $M.$  Let $U$ be a neighborhood of $p$ diffeomorphic to $\mathbb{R}^4$
  via a diffeomorphism which sends $p$ to $0 \in \mathbb{R}^4$
 The manifold $\widehat{M}$ obtained by removing  $p$ from $U$ and identifying $U \setminus \{p\}$
 with either $\tau_{\mathbb{C}P^1}^* \setminus \mathcal{Z_{\tau^*}}$ or with 
 $\tau_{\mathbb{C}P^1} \setminus \mathcal{Z_{\tau}}$ is called a  topological blow-up of $M$ along 
 $p.$
 \end{definition}
 
\begin{remark}\mbox{}
\begin{enumerate}[label=(\alph*)]
\item  The operation of topological blow-up of a manifold along a point corresponds to its connected
sum with $\mathbb{C}P^2$ or $\overline{\mathbb{C}P^2}.$ While performing 
a topological blow-up, if we use the tautological
line bundle $\tau_{\mathbb{C}P^1},$ then  we get 
$M \# \overline{\mathbb{C}P^2}.$ On the other hand, if  we use the dual bundle to $\tau_{\mathbb{C}P^2}$, then  we get 
$M \# \mathbb{C}P^2.$

\item Topological blow-up of $M$ along $p$ produces a manifold $\widehat{M}$ admitting an embedded $\mathbb{C}P^1$ with self intersection number $\pm 1.$ Recall that
the usual blow-up always produces an embedded $\mathbb{C}P^1$ with self intersection $-1.$ 

\item Throughout this discussion,  an embedded $\mathbb{C}P^1$ in a $4$--manifold $M$ with self intersection number $\pm 1$ will be called an \emph{exceptional sphere} in $M.$

\end{enumerate}
\end{remark}

\begin{definition}[Topological blow-down] Let $\widehat{M}$ be a smooth $4$--manifold admitting an
embedded $\mathbb{C}P^1$ whose normal bundle is isomorphic to  $\tau_{\mathbb{C}P^1}$ or $\tau_{\mathbb{C}P^1}^*.$ That is  the embedded $\mathbb{C}P^1$ is an exception sphere in 
$\widehat{M}.$  In this case, we can carry out the process exact opposite of the one describe in the 
definition of blow-up, where we remove a tubular neighborhood of $\C P^1$ and replace it
with a $4$-ball. The resulting manifold $M$ that we obtain as a result of this process is 
called a topological  blow-down of $\widehat{M}.$
\end{definition}

\begin{remark} \mbox{}
\begin{enumerate}
%\item Observe that the Lefschetz pencil $\pi:M \setminus B \rightarrow
%\math\itebb{C}P^1$ naturally  extends to a Lefschetz fibration on the blow-up $\widehat{M} =
%M \#_p \overline{\C P^2}.$
%
\item Observe that given a manifold $M$ admitting an embedding $\mathbb{C}P^1$ with 
its self intersection number $\pm 1,$ we can perform topological blow-down operation.

\item Suppose we are given a manifold $M \# \mathbb{C}P^2\# \overline{\mathbb{C}P^2}.$
Let $E_{1}$ and $E_{-1}$ be two embedded $\mathbb{C}P^1$'s corresponding to zero 
sections of $\tau_{\mathbb{C}P^1}^*$ and $\tau_{\mathbb{C}P^1}$ respectively. Suppose we have
$f: M \# \mathbb{C}P^2\# \overline{\mathbb{C}P^2} \rightarrow \mathbb{C}P^1$ be a SBLF such that
the intersection number of a fiber with $E_{1}$ is $1,$ and the intersection with $E_{-1}$ is $-1,$ then
the two operations of blow-downs corresponding to removal of $E_{-1}$ and $E_{1}$ produces 
a generalized SBLP on $M$

\end{enumerate}
\end{remark}

\subsection{Construction of SBLF on $M \# \mathbb{C}P^2 \# \overline{\mathbb{C}P^2}$}\mbox{}

The purpose of this subsection is to establish a SBLF on 
$M \# \mathbb{C}P^2 \# \overline{\mathbb{C}P^2}$ which satisfies the property that intersection 
of each fiber with  two exceptional spheres $E_1$ and $E_{-1}$  corresponding to 
zero sections is $-1$ and $+1$ respectively.

\begin{lemma}\label{lem:existence_of_special_SBLF}
Consider a closed orientable smooth manifold $M \# \mathbb{C}P^2 \# \overline{\mathbb{C}P^2}.$ 
There exists a SBLF  
$f:M \# \mathbb{C}P^2 \# \overline{\mathbb{C}P^2} \rightarrow \mathbb{C}P^1$ which satisfies the following:

\begin{enumerate}
\item The lower genus fiber  has its genus bigger than $1.$
\item The fibration agrees with the standard fibration in a tubular neighborhood of both
exception spheres $E_1$ and $E_{-1}.$
\end{enumerate}

In particular, bowling down the SBLF $f:M \# \mathbb{C}P2 \# \overline{\mathbb{C}P^2} \rightarrow \mathbb{C}P^1$ produces a generalized SBLP on $M.$
\end{lemma}

In  \cite[Theorem:6.5]{BO1},  I. Baykur and O. Saeki established the existence of simplified broken 
Lefschetz pencil  for any near symplectic manifold for any near symplectic manifold admitting
connected singular locus for near symplectic structure.  It is easy 
to see that following the  proof of \cite[Theorem:6.5]{BO1} -- essentially verbatim -- provides a
proof of Lemma~\ref{lem:existence_of_special_SBLF}.

\begin{proof}
To begin with, notice that there exists  an embedded surface $\Sigma$  in 
$M \# \mathbb{C}P^2 \# 
\overline{\mathbb{C}P^2}$ which satisfy the following properties:

\begin{itemize}
\item The self intersection of $\Sigma$ is 0
\item $\Sigma \cap E_{1} = +1$ and $\Sigma \cap E_{-1} = -1.$
\item $\Sigma$ is connected and the genus of Sigma is bigger than three. 
\end{itemize}

Observe that  since the self inter section of $E_1$ is $+1$ and $E_{-1} = -1$ it is easy to construct
a disconnected surface consisting of disjoint union of two spheres. By making connected sums of
these two sphere with an embedded surface  bounding a $3$--dimensional handle-body and
embedded in $B^4,$ it is easy to construct such a surface. 
    
Consider the map $\pi: \Sigma \times \mathbb{D}^2 \rightarrow \mathbb{D}^2,$ corresponding to 
the projection on  the  second factor, and regard $\mathbb{D}^2$ as embedded in $\mathbb{C}P^1$ as  a southern hemisphere.  This allows us to  regard $\pi$ as a map from a tubular neighborhood 
    $\mathcal{N} (\Sigma)$ to southern hemisphere. Next, construct  map to a map from
    $g:\mathcal{N}(\Sigma) \cup \mathcal{N}(E_1) \cup \mathcal{N}(E_{-1}) \rightarrow \mathbb{C}P^1$ which satisfies the following:
    properties:
\begin{enumerate}
\item The map when restricted $\mathcal{N}(E_1)$ and $\mathcal{N}(E_{-1}$ is the surjection 
on $\mathbb{C}P^1$ coming from the bundle projections $\pi_{E_1}: \mathcal{N}(E_1) \rightarrow
E_1,$ and $ \pi_{E_{-1}}: \mathcal{N}(E_{-1}) \rightarrow E_{-1}.$
\item  The map agrees with $\pi$ when restricted to $\mathcal{N}(\Sigma).$
\end{enumerate}    

Next, extend the map $g$ to a generic smooth map 
$\widehat{f}: M \# \mathbb{C}P^2 \# \overline{\mathbb{C}P^2} \rightarrow \mathbb{C}P^1.$  
According  to \cite[Remark:4.5]{BO1}, this map can be modified to produce a SBLF 
$\widehat{f}:M \# \mathbb{C}P2 \# \overline{\mathbb{C}P^2} \rightarrow \mathbb{C}P^1$ such that
all the modification performed while obtaining the SBLF from $g$ are performed alway from 
the region where $g$ is defined. 

Next, we convert the SLBF $\widehat{f}:M \# \mathbb{C}P^2 \# \overline{\mathbb{C}P^2} \rightarrow \mathbb{C}P^1$ to an SLBF 
$f:M \# \mathbb{C}P2 \# \overline{\mathbb{C}P^2} \rightarrow \mathbb{C}P^1$  whose lower genus fiber is bigger than $3$ by applying the technique from the proof of Theorem~\ref{thm:existence_SBLF}.  The SLBF $f:M \# \mathbb{C}P2 \# \overline{\mathbb{C}P^2} \rightarrow \mathbb{C}P^1$ can be ensured to satisfy the required because every fiber of
$f$ is homologous to the original fiber $\Sigma$ and hence the intersection of fibers of
$f$ has same property that $\Sigma$ had. This completes our argument.

\end{proof}

\begin{remark}\label{rmk:notation_special_SBLF}\mbox{}

\noindent From now on the SBLF  on $M \# \ C P^2 \# \overline{\ C P^2}$described in the statement of Lemma~\ref{lem:existence_of_special_SBLF} will be denoted by the notation $\pi_{spl}: M \# \C P^2 \# \overline{\C P^2} \rightarrow \C P^1. $
\end{remark}

\subsection{Blow-up and blow-down of $\C P^3$ along $\C P^1$}\mbox{}

Let us begin this sub-section by  making  a convention. By a standard
$\mathbb{C}P^1$ in $\overline{\mathbb{C}P^2},$ we mean a $\mathbb{C}P^1$ embedded
in $\overline{\mathbb{C}P^2}$ with its normal bundle isomorphic to the dual of the tautological
line bundle over $\mathbb{C}P^1.$ On the other hand, by a standard $\mathbb{C}P^1$ in
$\mathbb{C}P^n,$ we mean 
$\{[z_1,z_2, \cdots ,z_n] | z_i = 0 \hspace{0.1cm} \forall \hspace{0.1cm}  i \geq 3\},$ where $[z_1,
z_2, \cdots, z_n]$ denotes the homogeneous coordinates of $\mathbb{C}P^n.$

Consider $\C P^3$ and a standard $\C P^1$ embedded in it.  Fix 
a local trivialization $\mathbb{D}^2 \times \C^2$  of the normal bundle $\mathcal{N} (\C P^1)$ of $\C P^1$
in $\C P^3$.  

Now consider $\mathbb{D}^2 \times \C^2 \times \C P^1$ and a subset $V$ of 
$\mathbb{D}^2 \times \C^2 \times \C P^1 $ given by, 

$$V = \{(w, z_1, z_2, l) | \hspace{0.15cm} \|z_1^2\| + \| z_2^2 \| \leq 1 \hspace{0.15cm} \text{and}
\hspace{0.15cm} (z_1, z_2) \in l \},$$ 

where a point $l$ in $\C P^1$ is identified with the complex linear subspace  corresponding to that point. 

Now, observe that the complement of $\mathbb{D}^2 \times \{(0,0)\} \times \C P^1$ in $V$ can be identified with the complement of $\mathbb{D}^2 \times \{(0,0)\}$ in
$\mathbb{D}^2 \times \C^2.$

Choose two local trivializations  $U_1 \times \mathbb{C}^2$ and $U_2 \times \mathbb{C}^2$ over open set $U_1$ and $U_2$
such that $U_1$ and $U_2$ cover $\C P^1.$  By the (topological) blow-up of $\C P^3$ along $\C P^1$
we mean the operation of removing $U_i \times \{(0,0)\}$ from $U_i \times \C^2,$ for each $i,$ and
replacing it with the interior of $V$ as discussed in the previous paragraph. 

\begin{remark}\mbox{}

\begin{enumerate}
\item First of all, observe  that since the real normal bundle of $\mathbb{C}P^1$ in $\mathbb{C}P^3$ 
is trivial,  the manifold $\mathcal{B}_{\C P^1}(\C P^3)$ is
diffeomorphic to $\C P^1\times \C P^2.$ 

\item \emph{Exceptional divisor} of $\mathcal{B}_{\mathbb{C}P^1}(\mathbb{C} P^3)$ is the union of $\mathbb{D}^2\times\lbrace (0,0)\rbrace \times \mathbb{C}P^1$ over a finite collection $V_s$ of trivializations of the bundle $\mathcal{N}(\C P^1)$. Again notice that the triviality of the normal bundle 
of $\mathbb{C}P^1$ in $\mathbb{C}P^3$ implies   that the exceptional divisors is diffeomorphic to  $\mathbb{C}P^1\times \mathbb{C}P^1.$ 

\end{enumerate}
\end{remark}

 The  notion of blow-up discussed above is a  particular case of blow-up of a 
manifold along a sub-manifold. We refer \cite[p. 196,602]{GH} for a detailed
discussion on blow-ups.

By a blow-down of $\mathcal{B}_{\C P^1}(\C P^3)$ we will mean the process  exactly
opposite to the process of blow-up.  More precisely, let $\mathcal{B}_{\C P^1}(\C P^3)$
be obtained by blowing up a $\C P^1.$ Let  $E$ be the exceptional divisor obtained
as a result of the blow-up.  By blow-down of $\mathcal{B}_{\C P^1}(\C P^3)$, we mean 
removal of a tubular neighborhood of $E$ and replacing it by a tubular neighborhood
of $\C P^1$ in $\C P^3.$

%
%of removing the interior of $V$ and replacing it by $\mathcal{N}( \C P^1)$ to obtain $\C P^3 $ from $\mathcal{B}_{\C P^1} (\C P^3).$ Since $\C P^2 \times
%\C P^1$ is diffeomorphic to $\mathcal{B}_{\C P^1}(\C P^3)$ the following process
%will also be termed as  (topological) blow-down:
%%
%Consider $\C P^1 \times \C P^1$ embedded in $\C P^2 \times \C P^1$ via the
%standard embedding of $\C P^1$ in $\C P^2.$ Observe that a tubular neighborhood
%$\mathcal{N}(\C P^1 \times \C P^1)$ in $\C P^2 \times \C P^1$ is diffeomorphic to
%$V$ by a diffeomorphism that send $\C P^1 \times \C P^1$ contained in $\mathcal{N}(\C P^1 \times \C P^1)$ to the exceptional divisor $\C P^1 \times \overline{\C P^1}$ 
%contained in $V.$ Hence we can remove a tubular neighborhood of $\mathcal{N} (\C P^1
%\times \C P^1)$ form $\C P^2 \times \C P^1$ and replace it by 
%$\mathcal{N}(\C P^1)$ to get $\C P^3$ from $\C P^2 \times \C P^1.$ 
%
%

We  say that $\C P^3 $ is obtained from $\mathcal{B}_{\C P^1}(\C P^3)$ by blowing down along $E.$ Since $E$ is diffeomorphic to $\C P^1\times\C P^1,$ we sometimes do not distinguish between
$E$ and $\C P^1 \times \C P^1$ and say that $\C P^3$ is obtained from $\mathcal{B}_{\C P^1}(\C P^3)$ by blowing down along $\C P^1 \times \C P^1.$

We end this subsection with the following:

\begin{lemma}\label{lem:SBLP_embedding}
Let $M \# \C P^2 \#  \overline{\C P^2}$ be a smooth manifold. Let
 $\pi_{spl}: M \# \C P^2 \# \overline{\C P^2} \rightarrow \C P^1$ be 
the  SBLF on $M \# \C P^2 \#  \overline{\C P^2}$ as in the statement of Lemma~\ref{lem:existence_of_special_SBLF}  . If there exists a SBLF embedding of 
$M \# \C P^2 \# \overline{\C P^2}$ in 
 $\mathcal{B}_{\mathbb{C}P^1}(\C P^3)$ such that  each fiber of SBLF intersects the standard 
 $\C P^1$ of the fiber $\C P^2$ of $\mathcal{B}_{\C P^1} (\C P^3)$ in two distinct but fixed points, then there exist an embedding of $M$ in $\C P^3$ such that the standard  pencil of $\C P^3$ induces the  generalized SBLP of $M$ corresponding to the SBLF  of $M \# \C P^2 \# \overline{\C P^2}$
\end{lemma}

\begin{proof} 
Let $E_1$ and $E_{-1}$ be two exceptional divisors of $M \# \C P^2 \overline{\C P^2}.$  
Recall the exceptional divisor of $\mathcal{B}_{\C P^1}(\ C P^3)$ consist of a  union of two
local exceptional divisors of the type $U_i \times  W,$ where $W \subset \C P^1 \times \C^2$
consist of $\{(l, z_1, z_2)| (z_1, z_2) \in l\}.$ Since by hypothesis the fiber of $\pi_{spl}$ intersects
the standard $\C P^1$ inside $\C P^2$ in a pair of fixed point, we can assume that
of a tubular neighborhoods of an exceptional divisors $E_{\pm 1}$ is contained in $U_i \times W,$
and since the embedding is fiber preserving  it consist of $\{p_{\pm} \} \times W \subset U_1\times W.$

Furthermore, by the definition of the blow-up, the fibration on $\mathcal{B}_{\C P^1}(\C P^3)$ restricted to  $U_1 \times W $ can be assumed to be given by $(u, l, z_1, z_2) \rightarrow l.$ This clearly implies the when we blow-down $\mathcal{B}_{\C P^1}(\C P^3)$ along the exceptional
divisor $\C P^1 \times \C P^1$ we get $M \subset \C P^3$ with standard pencil of $\C P^3$ inducing 
the generalized SBLP on $M$ associated to SBLF $\pi_{spl}:M \# \C P^2 \overline{\C P^2} \rightarrow \C P^1.$ 

\end{proof}

\subsection{Embeddings in $\mathcal{B}_{\C P^1}(\C P^3)$}\mbox{}

In this sub-section we establish SLBF embedding of the special SLBF $\pi_{spl}: M \# \C P^2 \# 
\overline{\C P^2} \rightarrow \C P^1$ in $\mathcal{B}_{\C P^1}(\C P^3).$

 \begin{figure}[H]
 \includegraphics[scale=.8]{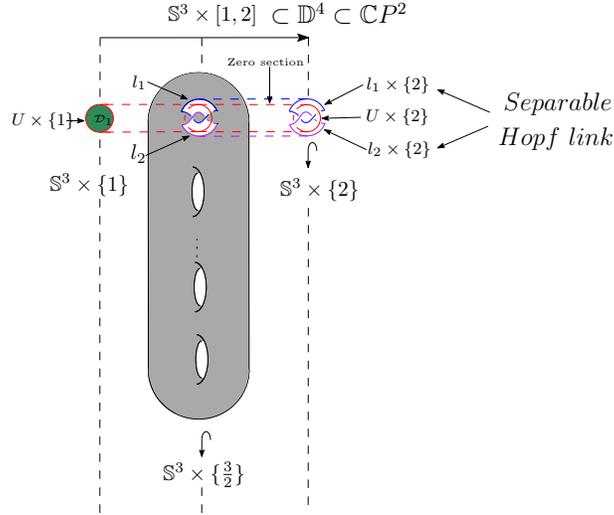}
 \caption{ Figure  depicts an embedded surface  in $\mathbb{C}P^3$ which is flexible and 
 in a standard position. The diagram focus on a collar $\mathbb{S}^3 \times [1,2]$ of a 
 $4$--ball $\mathbb{D}^4$ regarded as the unique zero handle $H_0$ of $\mathbb{C}P^2.$ The circle
 $U$ is the attaching circle of the unique $2$--handle $H_2$. $U\times [1,2]$ with the core
 disk attached at $U\times \{2\}$ and the green disk at $U\times \{1\}$ forms the 
 standard $\mathbb{C}P^1$ embedded in $\mathbb{C}P^2$}
 \label{fig:kirby_diagram_CP^2}
 \end{figure}

\begin{proposition}\label{prop:embedding_in_CP^2_times_CP^1}
Let $M$ be a closed orientable smooth $4$-manifold. Let $f:M \rightarrow \mathbb{C}P^1$ be a SBLF with the lower genus fiber having genus bigger than $1.$ There exists a SLBF embedding of
$M$ in $\mathcal{B}_{\C P^1}(\C P^3)$ such that each fiber of $SLBF$ intersect 
the standard $\mathbb{C}P^1$ in the fibre $\mathbb{C}P^2$ in a pair of cancelling intersection points.
\end{proposition}

\begin{proof}  We will follow the line of argument  we used to establish Theorem~\ref{thm:embedding_in_CP^2_times_CP^1}. The only 
difference is that the fibration $\mathcal{B}_{\C P^1}(\C P^3)$ is not a trivial fibration. However,
since the normal bundle of $\C P^1$ in $\C P^3$ is trivial, we note that this bundle as a real
bundle is trivial  provided we remove the section of the fiber bundle $\mathcal{B}_{\C P^1}(\C P^3) \rightarrow \C P^1$ corresponding to the exceptional divisor.  Hence, we first consider neighborhoods of exceptional divisors $E_1$ and $E_{-1}$ of $M \# \C P^2 \# \overline{\C P^2},$ and embed them in a tubular neighborhood of the exceptional divisor 
$\C P^1 \times \C P^1$ of $\mathcal{B}_{\C P^1}(\C P^3)$ such that the embedding is
fiber preserving.   

In order to produce this embedding recall that a tubular neighborhood of the  exceptional divisor $\C P^1 \times \C P^1$ is
union of two open sets  $U_i \times W, i = 1,2.$  Consider $U_1 \times W,$ and  let us
denote by $\pi$ the fibration $\pi: \mathcal{B}_{\C P^1}(\C P^3) \rightarrow \C P^1$ obtained
via blow-up of the standard pencil of $\C P^3.$

Next,  consider a pair of points $p_{+},$ $p_{-}$ in $U_1,$ and consider 
spheres $\{p_{\pm}\} \times \C P^1$ embedded in $U_1 \times W$ Since tubular neighborhood
of $E_{\pm 1}$ is isomorphic to tubular neighborhood of any sphere in $U_1 \times W$ of 
the form $\{p\} \times \C P^1,$ where $p$ is a point in $U_1,$  we get the there exist
an embedding of small neighborhoods of $E_{\pm 1}$ in a neighborhood of the exceptional 
divisor $\C P^1 \times \C P^1$ such that $\pi_{spl}$ restricted to this neighborhood agrees
with restriction of $\pi$ on the embedded neighborhoods.

Observe that the intersection of the embedded neighborhoods of $E_{\pm 1}$ with a fiber
of the fibration $\pi: \mathcal{B}_{\C P^1} (\C P^3)$ is a pair of disk satisfying the property that
the intersection of this pair of disk with the boundary of a small tubular neighborhood of
$\C P^1 \subset \C P^2$ is a Hopf link. Furthermore, observe that the since the embedding
of the neighborhood of $E_1$ with  tubular neighborhood of $\{p_{+}\} \times \C P^1$ is
orientation reversing, and the embedding of neighborhood of $E_{+1}$ with $\{p_{-} \times \C P^1$
is orientation preserving. This implies that if we establish the following:

\begin{enumerate}
\item $\C P^2$ admits a separable Hopf link,
\item there exists an embedding of any surface of genus $g$ in $\C P^2$ which is
standard embedding,
\item the embedded  surface $\Sigma_g$ intersects the standard $\C P^1$ contained in $\C P^2$ in a pair of algebraically cancelling point, and $\Sigma_g \cap \partial \mathcal{N}(\C P^1)$ 
is a Hopf link in $\partial \mathcal{N}(\C P^1),$ where 
$\mathcal{N}(\C P^1)$ is a fixed open tubular neighborhood of $\C P^1$ in $\C P^2,$  
\end{enumerate}

\noindent  then the triviality of the fibration $\pi: \mathcal{B}_{\C P^1}(\C P^3)$ 
an argument similar  to the one which establishTheorem~\ref{thm:embedding_in_CP^2_times_CP^1} implies required  SBLF embedding of $M \# \C P^2 \# \overline{\C P^2}$ in 
$\mathcal{B}_{\C P^1}(\C P^3).$

Hence, the task at our hand is to establish an embedding of a surface satisfying the three properties
listed above. We now proceed to produce such an embedding.

 To begin with, we regard $\overline{\mathbb{C}P^2}$ as a handle-body with 
the $0$--handle $H_0$ corresponding to $B^4(0,2)$ -- the $4$--ball of radius $2$ in $\mathbb{C}^2$ with its center at the origin -- to which a $2$--handle $H_2$ is attached along an unknot with framing $+1.$ Finally a $4$--handle $H_4$ is attached  to the $4$--manifold, which is the union of the $0$--handle $B^4(0,2)$ and the $2$--handle $H_2$. Regarding $H_0$ as a ball.  Let 
$\mathbb{S}^3 \times [1,2]$ be a collar of $\partial H_0.$ Let $U \times \{2\}$ be the 
attaching circle of $H_2.$  Observe  that
any Hopf link consisting of a parallel copy of the attaching circle -- say $l_1 \times \{2\}$ and a 
circle $l_2 \times \{2\}$
which links both the attaching circle and $l_1$ once as depicted in Figure~\ref{fig:kirby_diagram_CP^2} constitute a Hopf link which is separable.  This is because
$l_1 \times \{2\}$ bounds a parallel copy of the core of $2$--handle, and $l_2 \times \{2\}$ bound
a disk in the unique $4$--handle.

Next, consider cylinders $l_i \times [\frac{3}{2}, 2], i = 1,2.$ They intersect 
$\mathbb{S}^3 \times \{\frac{3}{2}\}$ in $ l_i \times \{\frac{3}{2}\}.$  Observe that there exists
a surface $\Sigma_g$ with two boundary component whose boundary is  the Hopf link 
$l_1 \times \{\frac{3}{2}\} \sqcup l_2 \times \{\frac{3}{2}\}.$ 
See Figure~\ref{fig:kirby_diagram_CP^2}. It follows from an argument similar to the one used
in establishing Lemma~\ref{lem:embedding_stand_flexible} that the embedding is both flexible 
and in a standard position. 

Regarding the standard $\C P^1$ as the union of core of $2$--handle $H_2$ with a disk $\mathbb{D}$ that $U \times \{2\}$ bounds, we see that the embedded $\Sigma_g$ intersects $\mathbb{C}P^1$ in 
a pair of points. This pair has to be algebraically cancelling as we can push the disk $\mathbb{D}$
down to produce an isotopy of $\mathbb{C}P^1$ that sends the $\mathbb{C}P^1$ to a new
$\mathbb{C}P^1$ which consist of union of core of $H_2,$ $U \times [1,2],$ and 
a disk $\mathbb{D}$ that $U \times\{ 1\}$ bounds. The disk that $U \times \{1\}$ bounds 
is denoted by a blue disk in Figure~\ref{fig:kirby_diagram_CP^2}. Notice that the isotoped $\mathbb{C}P^1$ is disjoint from $\Sigma_g$ implying that the algebraic intersection of $\Sigma_g$ with the standard $\mathbb{C}P^1$ is zero.

This completes our argument.

\end{proof}

Now we have established all the results necessary to establish Theorem~\ref{thm:embedding_4-manifolds}. We now proceed and supply a proof of Theorem~\ref{thm:embedding_4-manifolds}.

\subsection{Proof of Theorem~\ref{thm:embedding_4-manifolds}}
Recall that we need to prove that every smooth orientable closed $4$--manifold admits an
embedding in $\C P^3.$

\begin{proof}[Proof of Theorem~\ref{thm:embedding_4-manifolds}]

Let $M$ be the given closed orientable $4$--manifold. Consider the manifold 
$\widehat{M} =M \# \mathbb{C}P^2 \# \overline{\mathbb{C}P^2}$ thought as a blow-up of $M$ done at two distinct
points $p_1$ and $p_2.$  Recall  that $\widehat{M}$ admits a pair of 
exceptional divisors 
-- say $E_1$ and $E_{-1}$  such that $E_1 \cap E_1 = 1$ while $E_{-1} \cap E_{-1} = -1.$

Next, apply Lemma~\ref{lem:existence_of_special_SBLF} to produce a SBLF on 
$\widehat{M}$ which satisfies the following:

\begin{enumerate}
\item The lower genus fiber  has its genus bigger than $1.$
\item The fibration agrees with the standard fibration in a tubular neighborhood of both
exception spheres $E_1$ and $E_{-1}.$
\end{enumerate}

Now, by Proposition~\ref{prop:embedding_in_CP^2_times_CP^1} there exist SBLF  embedding of
$\widehat{M}$ in $\mathbb{C}P^2 \times \mathbb{C}P^1.$ 
Since $\C P^2 \times \C P^1$ is diffeomorphic to $\mathcal{B}_{\C P^1} (\C P^3),$ 
we get an embedding of $M \# \C P^2  \# \overline{\C P^2}$ in $\mathcal{B}_{\C P^1}( \C P^3).$

Also notice that the intersection property of the embedded fiber of SLBF with 
with standard $\mathbb{C}P^1$ contained in $\mathbb{C}P^2$ stated in  Proposition~\ref{prop:embedding_in_CP^2_times_CP^1} 
implies that the  embedding is such that each fiber of the SBLF associated to 
$M \# \C P^2 \# \overline{\C P^2}$ intersects  the standard $\mathbb{C}P^1$ of a fiber $\mathbb{C}P^2$ of the trivial fibration  $\mathcal{B}_{\C P^1} (\C P^3) \rightarrow \mathbb{C}P^1$  in a  pair of 
algebraically cancelling points.  

Finally, blow-down $\mathcal{B}_{\mathbb{C}P^1}(\C P^3)$ along its exceptional divisor. Observe that
Lemma~\ref{lem:SBLP_embedding} implies that blow-down produces an embedding
of $M$ in $\C P^3$ such that the standard Lefschetz pencil of $\C P^3$ induces 
a SBLP on $M.$

\end{proof}

\bibliographystyle{amsplain}

\end{document}